\documentclass[11pt,letterpaper]{amsart}

\usepackage{mathtools}
\usepackage{xcolor}
\usepackage{amssymb}
\usepackage[utf8]{inputenc}
\usepackage[cyr]{aeguill}
\usepackage[pdfdisplaydoctitle=true,
            colorlinks=true,
            urlcolor=blue,
            citecolor=blue,
            linkcolor=blue,
            pdfstartview=FitH,
            pdfpagemode= UseNone,
            bookmarksnumbered=true]{hyperref}
\usepackage{todonotes}
\usepackage{enumitem}
\usepackage{mathtools}
\mathtoolsset{showonlyrefs=true}

\usepackage{geometry}
\newtheorem{theorem}{Theorem}

\newtheorem*{theorem*}{Theorem}
\newtheorem{lemma}{Lemma}[section]
\newtheorem{corollary}[lemma]{Corollary}
\newtheorem*{corollary*}{Corollary}
\newtheorem{proposition}[lemma]{Proposition}
\theoremstyle{definition}

\newtheorem{conjecture}{Conjecture}
\newtheorem{remark}[lemma]{Remark}

\newcommand{\spacey}{\mathcal{P}}
\newcommand{\decayspace}{\mathcal{X}}

\newcommand{\R}{\mathbb R}
\newcommand{\Laplacian}{\Delta}
\newcommand{\OvDel}{\overline{\Laplacian}}
\newcommand{\Diff}{\operatorname{Diff}}
\newcommand{\aconst}{\sigma}
\newcommand{\A}{(\aconst-\Delta)^k}
\newcommand{\ad}{\operatorname{ad}}
\newcommand{\Ad}{\operatorname{Ad}}
\newcommand{\funca}{\alpha}
\newcommand{\funcb}{\beta}
\newcommand{\ratioa}{R_{\alpha}}
\newcommand{\ratiob}{R_{\beta}}
\newcommand{\jfunc}{j_n}

\newcommand{\constyp}{c_p}
\newcommand{\numerator}{N}
\newcommand{\ubar}{\overline{u}}
\newcommand{\id}{\mathrm{id}}
\newcommand{\Jac}{\operatorname{Jac}}
\newcommand{\omegabar}{\overline{\omega}}
\newcommand{\Pspace}{\mathcal{W}}
\newcommand{\vectyF}{\mathcal{F}}
\newcommand{\vectyFone}{\mathcal{F}^1}
\newcommand{\vectyFtwo}{\mathcal{F}^2}

\newcommand{\power}{\lambda}

\title[Breakdown of Euler-Arnold equations]{Liouville comparison theory for breakdown of Euler-Arnold equations}

\author[]
{Martin Bauer, Stephen C. Preston, and Justin Valletta}

\address{Martin Bauer: Florida State University and University of Vienna\\
Stephen Preston: Brooklyn College and CUNY Graduate Center\\
Justin Valletta: Florida State University}
\email{bauer@math.fsu.edu, stephen.preston@brooklyn.cuny.edu, jvalletta@fsu.edu}

\date{\today}
\subjclass[2010]{%
}

\begin{document}

\begin{abstract}
In this article we study breakdown of solutions for (generalized) Euler-Arnold equations on $\mathbb R^n$. 
Our method is based on treating the equation in Lagrangian coordinates, where it is an ODE on the diffeomorphism group, and  comparison  with the Liouville equation; in contrast to the usual comparison approach at a single point, we apply comparison in an infinite dimensional function space. We thereby show that the Jacobian of the Lagrangian flow map of the solution reaches zero in finite time, which corresponds to $C^1$-breakdown of the velocity field solution. We demonstrate the applicability %
of our result by proving breakdown of smooth solutions to some higher-order versions of the
EPDiff equation in all dimensions $n\geq 3$, 
even in situations where the one-dimensional equation has global solutions, such as the EPDiff equation corresponding to a Sobolev metric of order two. 
\end{abstract}

\maketitle

\setcounter{tocdepth}{1}
\tableofcontents

\section{Introduction}
\subsection{Background}
In his seminal paper~\cite{arnold2014differential} Arnold recast Euler's equation for the motion of an incompressible fluid as the geodesic equation of a (right invariant) Riemannian metric on the group of volume preserving diffeomorphisms. Since then it has been shown that an analogous geometric interpretation---as geodesic equations of a right-invariant connection on a group of diffeomorphisms---exists for many prominent PDEs in mathematical physics; such equations are now commonly referred to as (generalized) Euler-Arnold equations. Examples include the Camassa-Holm~\cite{camassa1993integrable,misiolek1998shallow,kouranbaeva1999camassa}, the Hunter-Saxton~\cite{hunter1991dynamics,lenells2007hunter}, the modified Constantin-Lax-Majda~\cite{constantin1985simple,escher2011geometry}, the KdV~\cite{ovsienko1987korteweg}, or the surface quasi-geostrophic equation (SQG)~\cite{washabaugh2016sqg}. See also~\cite{arnold2008topological,vizman2008geodesic} for further examples.

 In this article we will be interested the family of non-linear PDE of order $2k\in\mathbb R$ given by
\begin{equation}\label{eq:EPDiff}
\Omega_t+\nabla_U \Omega+  (\nabla U)^T\Omega+\operatorname{div}(U)\Omega=0,\qquad \Omega=(\aconst-\Delta)^k U, \quad \aconst\geq 0, k\geq 0
\end{equation}
where $U:[0,T)\times \mathbb R^n\to \mathbb R^n$ is a time dependent vector field and $\Delta$ is the vector Laplacian on $\mathbb{R}^n$.

The interest in this family of equations can be motivated from different angles. First, it encompasses several of the above mentioned physically relevant PDEs as special cases, including the Camassa-Holm, the Hunter-Saxton equations, and the mCLM equation. Second, it admits a geometric interpretation as an Euler-Arnold equation, i.e., it can be realized as the geodesic equation of the right-invariant  Sobolev metric of order $k$ on the group of diffeomorphisms of $\mathbb R^n$. Such equations are often also called EPDiff equations, short for Euler-Poincar\'e equation on the diffeomorphism group. Note that the terminology EPDiff was first used for the Euler-Arnold equation of the right invariant Sobolev metric of order one (corresponding to the Camassa-Holm equation) (see e.g.,~\cite{holm1998euler,holm2005momentum}), but  since then it has become common to use it more broadly for Euler-Arnold equations of general (Riemannian or Finslerian) metrics on groups of diffeomorphisms, see~\cite{mumford2013euler,bruveris2017completeness,bauer2015local,cotter2022r} and the references therein. Finally, right invariant Sobolev metrics on diffeomorphism groups and thus equation~\eqref{eq:EPDiff} take a central role in the fields of template matching and shape analysis~\cite{younes2010shapes,dryden2016statistical}, particularly in the widely acclaimed LDDMM framework~\cite{beg2005computing}. This approach, following the spirit of Grenander's pattern theory~\cite{grenander1996elements}, represents the differences between shapes as optimal transformations (diffeomorphisms) between objects, where  the optimality is measured with respect to a right invariant metric of high order on the transformation group. Consequently the EPDiff equation arises as the first order optimality condition. In this context, it is also known as the template matching equation.

Over the past decades the short and long-time existence of the EPDiff equation has been studied in detail. A significant amount of this analysis dates back to the seminal paper of Ebin and Marsden~\cite{ebin1970groups}, who used Arnold's  geometric picture~\cite{arnold2014differential} to obtain local well-posedness of the incompressible Euler-equations. Using similar methods, the local well-posedness for the EPDiff equation has been established, assuming that the order $k$ satisfies $k\geq \frac{1}{2}$ (independently of the dimension $n$); see~\cite{gay2009well,misiolek2010fredholm,escher2014right,bauer2015local,bauer2020well,trouve2005local}. Global existence can sometimes be established using entirely geometric arguments: if the order $k>\frac{n}{2}+1$ then the corresponding geometric description can be extended to a strong, right-invariant metric on a group of Sobolev diffeomorphisms. In this setting the global existence of solutions  holds true by general arguments, and one directly obtains the  global well-posedness in the smooth category by applying an Ebin-Marsden type of no-loss-no-gain argument; see e.g., the work by  Escher, Kolev, Michor, Mumford and others~\cite{bruveris2017completeness,escher2014geodesic,mumford2013euler,misiolek2010fredholm,bauer2015local,bauer2023regularity,ebin1970groups}.

For the EPDiff equation of order $k=0$ and $k=1$, much work has been dedicated to showing that the global wellposedness fails, i.e., that there exists smooth initial conditions such that the corresponding solutions break down in finite time. In particular in dimension $n=1$ and order $k=0$ this corresponds to breakdown of the inviscid Burgers' equation. For $n=1$ and $k=1$, we have to distinguish between the homogeneous inertia operator $A=-\partial_x^2$, corresponding to the Hunter-Saxton equation, and the non-homogeneous inertia operator $A=1-\partial_x^2$, corresponding to the Camassa-Holm equation. Breakdown (wave breaking) for the Camassa-Holm equation is known from the original paper of Camassa and Holm~\cite{camassa1993integrable}; see also~\cite{constantin1998wave} for a more rigorous mathematical analysis. The complete picture of the breakdown mechanism for this equation has been obtained by McKean~\cite{mckean2015breakdown}. For the Hunter-Saxton equation the situation is much simpler, as one can obtain an explicit solution formula~\cite{lenells2007hunter,bauer2014homogeneous} which then directly leads to proof of breakdown; see also~\cite{yin2004structure}. Note that this yields a complete characterization for (integer order) EPDiff equations in dimension one: they are globally well-posed for any smooth initial conditions if $k\geq 2$, and there exist initial conditions such that the corresponding solutions break down in finite time for $k\in \{0,1\}$. 

Studying breakdown for solutions to higher-dimensional EPDiff equations was first proposed by Chae and Liu  in ~\cite{chae2012blow}, where they also confirmed breakdown for the higher dimensional Burgers' equation $(k=0$). For the higher dimensional Camassa-Holm equation $(k=1)$, this result has been obtained by Li, Yu, and Zhai in~\cite{li2013euler}, who showed that the  breakdown mechanism of dimension one can be adapted to the higher dimensional situation.
 In the higher dimensional situation there is, however, a gap to the global well-posedness results for the EPDiff equation: already in dimension two the global existence result is only valid for $k\geq 3$; furthermore all the above breakdown results for EPDiff equations in higher dimensions are only for equations with $k=0$ or $k=1$ where the one-dimensional version already admits breakdown. The main results of the present article, which we will describe next, are of a  different nature: they include cases where the breakdown of the equation is a genuinely higher dimensional phenomenon; specifically we deal primarily with the case $k=2$ where the one-dimensional equation has global solutions for all initial conditions, while in dimensions $n\ge 3$ we will show that some solutions exhibit breakdown. Furthermore we suspect that the methods presented here can be adapted to any value of $k<\frac{n}2+1$ (including non-integer $k$) and in any dimension to produce similar results. Note that this would yield an almost complete characterization, as for $k>\frac{n}{2}+1$ it is known that global existence of solutions  holds true.

\subsection{Main Contributions}
The starting point of the present article is the observation that the EPDiff equation~\eqref{eq:EPDiff}
admits radial solutions, which allows us to reduce the investigations to an equation of a single variable.
The corresponding ``radial EPDiff equation'' takes the form
\begin{equation}\label{mainomega}
\omega_t + u \omega_r + 2u_r \omega + \tfrac{n-1}{r} u \omega = 0, \qquad \omega\partial_r = (\aconst - \Delta)^k (u\,\partial_r),
\end{equation}
where $\sigma \in \{0,1\}$ and $k\in \{1,2\}$ and where the vector Laplacian $\Delta$ on $\mathbb{R}^n$ for radial functions is explicitly given by
\begin{equation}\label{vectorlaplacian}
\Delta\big(u(r)\partial_r\big)=\Big(u''(r)  + \frac{n-1}{r}\, u'(r) - \frac{n-1}{r^2} \, u(r)\Big)\partial_r.
\end{equation}
Note that this is \emph{not} the same as the usual Laplacian on functions; in particular it is not the case that $\omega = (\aconst-\Delta)^k u$ as functions, which is why we must be careful in this notation.
 For simplicity we will simply denote the family of radial equations \eqref{mainomega} as either the radial $\dot{H}^k$ family if $\aconst=0$ or the radial $H^k$ family if $\aconst=1$, regardless of dimension $n$.

Our main breakdown results for this family of equations---thus also for the EPDiff equations~\eqref{eq:EPDiff} ---are summarized in the following theorem:
\begin{theorem*}[Theorems~\ref{exactsolnH1thm}, \ref{thm:H2dotbreakdown}, \ref{thm:H1breakdown} and \ref{thm:H2breakdown}]
Let $n\geq 3$ and suppose that the initial momentum  $\omega_0$ satisfies $\omega_0(r)\le 0$ for all $r\ge 0$.
Then the solution of the radial EPDiff equation~\eqref{mainomega} with inertia operator $(\sigma-\Delta)^k$, $k\in\{1,2\}$ and $\sigma\in\{0,1\}$ breaks down in finite time. If $k=1$ the result continues to hold for $n\leq 2$. 
\end{theorem*}
To be more specific, the case $k=1$ and $\sigma=0$ is shown  
in Theorem~\ref{exactsolnH1thm}, the case $k=2$ and $\sigma=0$  
in Theorem~\ref{thm:H2dotbreakdown}, the case $k=1$ and $\sigma=1$  
in Theorem~\ref{thm:H1breakdown}, and the case $k=2$ and $\sigma=1$  
in Theorem~\ref{thm:H2breakdown}.
The proof of Theorem \ref{exactsolnH1thm} is essentially a computation, while the proofs of Theorems \ref{thm:H2dotbreakdown}, \ref{thm:H1breakdown} and \ref{thm:H2breakdown} will follow from the more general breakdown result for Euler-Arnold type equations,
which is based on a comparison theorem  with a Liouville-type equation given below in Theorem \ref{Qtheorem}. Next we will explain the main ideas of it in more detail.


\subsection*{Reducing the EPDiff equation to an ODE}
Our approach is based on the fact that in Lagrangian coordinates, the EPDiff equation is an ordinary differential equation on a Banach space. Suppose $u(t,r)$ is a solution of the EPDiff equation \eqref{mainomega}; then the Lagrangian flow  $\gamma(t,r)$ is defined to be the solution of the flow equation $\gamma_t(t,r) = u(t,\gamma(t,r))$ with $\gamma(0,r)=r$, and it is a diffeomorphism as long as the velocity field $u(t,r)$ remains smooth (or at least $C^1$).
The basic principle is that every Euler-Arnold equation (and thus in particular the EPDiff equation) satisfies some analogue of the vorticity-transport law (i.e., momentum conservation), which for equation \eqref{mainomega} takes the form 
\begin{equation}\label{momentumtransportgeneral} 
\gamma(t,r)^{n-1} \gamma_r(t,r)^2 \omega(t,\gamma(t,r)) = r^{n-1} \omega_0(r).
\end{equation}
Inverting the operator $(\aconst - \Delta)^k$ that relates $\omega$ to $u$, we can solve for $u(t,\gamma(t,r))$ in terms of the initial momentum $\omega_0$ and the Lagrangian quantities $\gamma$ and $\gamma_r$. 
This yields the particle-trajectory form of the equations, cf.  Majda-Bertozzi~\cite{majda_bertozzi_2001} and Ebin~\cite{ebin1984concise} who used this method for proving local existence for the Euler equations of ideal fluid mechanics. For the radial EPDiff equation, this explicitly looks like 
\begin{equation}\label{eq:particle_equation}
\frac{\partial\gamma(t,r)}{\partial t} = \int_0^r \delta\big(\gamma(t,s),\gamma(t,r)\big) \, \frac{z_0(s)}{\gamma_s(t,s)} \, ds
+ \int_r^{\infty} \delta\big(\gamma(t,r),\gamma(t,s)\big) \, \frac{z_0(s)}{\gamma_s(t,s)} \, ds,\quad \gamma(0,r)=r
\end{equation}
for some Green function $\delta$ defined on the set
$$ D = \{ (r,s) \, \vert\, s\ge r \ge 0\} \backslash \{(0,0)\} \subset \mathbb{R}^2,$$
with $z_0(s) = s^{n-1}\omega_0(s)$.
Differentiating equation \eqref{eq:particle_equation} in $r$ gives a similar equation for $\gamma_r$. Remarkably this procedure often yields an ODE in the variables $(\gamma,\gamma_r)$, in the sense that the right side is a locally Lipschitz and in many cases smooth function on a Banach space, which can be used to construct a local existence theory using Picard iteration.

\subsection*{An explicit solution for the radial Hunter-Saxton equation}
 Our next ingredient is the existence of an explicit solution formula for the radial Hunter-Saxton equation, i.e., for \eqref{mainomega} with $\aconst=0$ and $k=1$, we can directly relate the solution to~\eqref{eq:particle_equation} to the Liouville type equation
\begin{equation}\label{liouvilleboundary}
\frac{\partial}{\partial t}
\ln{q(t,r)} = \int_r^{\infty} \frac{z_0(s)}{q(t,s)}\,ds, \qquad q(0,r)=1 \;\;\forall r\ge 0, \qquad \lim_{r\to\infty} q(t,r) = 1\;\;\forall t\ge 0.
\end{equation}
This is a well-known PDE that has both an ODE interpretation on a Banach space and a simple exact solution, found by Liouville in 1853~\cite{liouville1853equation}, which is of the form $q(t,r) = (1+ty_0(r))^2$ for some function $y_0$. Note that $q$ obviously  reaches zero in finite time if $y_0$ is ever negative. Sarria-Saxton~\cite{sarria2015sign} studied the global behavior of this equation in the periodic domain, inspired by their study of the Lagrangian flow approach to the Proudman-Johnson equation~\cite{sarria2013blow}, and this approach was further developed by Kogelbauer~\cite{kogelbauer2020global}.

The change of variables $q(t,r) = r^{1-n}\gamma(t,r)^{n-1}\gamma_r(t,r)$ turns the radial Hunter-Saxton equation into the Liouville equation \eqref{liouvilleboundary}. (Note that this $q$ is simply the Jacobian determinant of the radial diffeomorphism $\gamma$.) Thus we obtain the following explicit formula for the solution along with a precise breakdown criterion; {we refer to Theorem ~\ref{exactsolnH1thm} for a precise formulation of the assumptions}: 
 \begin{theorem*}
For any initial data $u_0$ the solution to the radial Hunter-Saxton equation (Equation~\eqref{mainomega} with $k=1$ and $\sigma=0$) satisfies the equation
\begin{equation*}%
\gamma(t,r)^{n-1} \rho(t,r) %
= r^{n-1} \left( 1 + \frac{t}{2}\left(u_0'(r)+\frac{n-1}{r}u_0(r)\right) \right)^2
, \qquad \frac{\partial \gamma}{\partial r}(t,r) = \rho(t,r).
\end{equation*}
with
\begin{equation*}
\frac{\partial \gamma}{\partial t}(t,r) = u\big(t, \gamma(t,r)\big), \qquad \gamma(0,r)=r.
\end{equation*}
The solution breaks down in finite time if and only if $u_0'(r)+\frac{n-1}{r}u_0(r) < 0$ for some $r\ge 0$.
\end{theorem*}
In dimension $n=1$ we recover the result of \cite{bauer2014homogeneous} for the Hunter-Saxton equation on the line, with odd initial data. In higher dimensions this leads to a proof of breakdown for the higher dimensional Hunter-Saxton equation, which has been studied by Modin~\cite{modin2015generalized} due to its connections to geometric statistics~\cite{khesin2013geometry}.

\subsection*{The general breakdown result}
Next we will describe the main methodological contribution of the article: the general breakdown result for equations of type~\eqref{eq:particle_equation} based on comparison theory for ODEs on Banach spaces. The main principle is to consider functions on the space $\Pspace$ consisting of continuous functions $q\colon [0,\infty)\rightarrow [a,b]\subset (0,\infty)$ and autonomous vector fields $F\colon \Pspace \to T\Pspace$ which are locally Lipschitz in the usual supremum norm topology. The space of such positive functions is a convex subset of a Banach space and has a partial order where $q\le \tilde{q}$ iff $q(r)\le \tilde{q}(r)$ for all $r\ge 0$; if $F$ and $\tilde{F}$ are vector fields on this space satisfying the monotonicity property $q\le \tilde{q} \Rightarrow F(q)\le \tilde{F}(\tilde{q})$, then the solutions of $q_t = F(q)$ and $\tilde{q}_t = \tilde{F}(\tilde{q})$ with the same initial condition $q(0)=\tilde{q}(0)$ will satisfy $q(t,r)\le \tilde{q}(t,r)$ for all $t\ge 0$ as long as the solution exists.

We will apply this principle to the Jacobian-like functions $q = Q(\gamma)\gamma_r$ depending on the Lagrangian flow $\gamma$, and we will use the right hand side of the Liouville equation~\cite{liouville1853equation} as our comparison function $\tilde F$. 
Our strategy will be to choose $Q$ to obtain an ODE for $q(t,r)$ and compare it to a known solution $\tilde{q}(t,r)$ of the Liouville equation which  approaches zero in finite time. This will then allow us to conclude that $\gamma_r(t,r)$ approaches zero at least as quickly and thus in finite time. Our main result of this part establishes a general breakdown result for equations of type~\eqref{eq:particle_equation}, {we refer to Theorem~\ref{Qtheorem} for a precise formulation of all assumptions:}
\begin{theorem*}
Assume that the kernel $\delta$ can be written in the form $\delta(r,s)=rs \varphi(r,s)$, where $\varphi$ is smooth
on $D=\{(r,s)\, \vert s\ge r\ge 0\} \backslash\{(0,0)\}$, such that for all $(r,s)\in D$ we have:
\begin{enumerate}%
\item $\varphi(r,s) > 0$; %
\item  $\frac{\partial^2}{\partial r\partial s} \ln{\varphi(r,s)} \ge 0; $ %
 \item there exists a $C>0$ such that %
\begin{equation}
S(r) := \frac{r \partial_1\varphi(r,r) \varphi(0,r)
- r \varphi(r,r) \partial_2\varphi(0,r)}{\varphi(0,r)^2}\ge C.
\end{equation}
\end{enumerate}
Suppose $z_0(r)\leq 0$ for all $r>0$.
Then there exists $T>0$ and $r_0\ge 0$ such that $\gamma_r(T,r_0)=0$, i.e., the solution $\gamma$ leaves the group of $C^1$ diffeomorphisms, and thus the $C^1$ norm of the velocity field $u$ blows up at finite time $T$.
\end{theorem*}

%

%
%

\subsection*{Geodesic completeness of the geometric interpretation:} As a corollary of our results and the global existence results described above we obtain the following complete characterization of geodesic completeness (incompleteness, resp.) of the corresponding geometric picture for integer order Sobolev metrics in dimension three:
\begin{corollary*}[Corollary~\ref{cor:geodesiccomplete}]
The diffeomorphism group $\operatorname{Diff}(\mathbb R^3)$ equipped with the right-invariant Sobolev metric of order $k\in \mathbb N$ is geodesically complete if and only if $k\geq 3$, i.e., for any $k\geq 3$ and any initial conditions $U_0\in H^{\infty}(\mathbb R^3,\mathbb R^3)$ the solution to the geodesic equation (EPDiff equation, resp.) exists for all time $t$, whereas for any $k\in \{0,1,2\}$ there exists initial conditions $U_0\in H^{\infty}(\mathbb R^3,\mathbb R^3)$ such that the solution blows up in finite time.
\end{corollary*}
It is likely that a similar statement is true on $\mathbb{R}^n$ and for fractional values of $k$; see Section \ref{sec:conclusion}.

\subsection{Structure of the Article}
In Section~\ref{sec:background} we present some general mathematical background. We begin with the 
solution of the Liouville equation, and we present the Green functions of the Laplace operator acting on radial functions. Finally we present a comparison theorem for ODEs in Banach spaces, which will be one of the main ingredients for our breakdown theorem. In Section~\ref{sec:EulerArnold} we describe the geometric picture for the EPDiff equation, derive the momentum transport law, write its particle formulation, and prove the existence of radial solutions. We also discuss the local well-posedness theory for both smooth and ``classical'' solutions along with the basic breakdown result. Next, in Section~\ref{sec:radialHS}, we derive an explicit solution formula for the radial Hunter-Saxton equation, i.e., the EPDiff equation for the $\dot{H}^1$ metric with radial initial conditions in any dimension. We then present the main breakdown result and its proof in Section~\ref{sec:general}, which we apply  to show breakdown of various EPDiff equations ($\dot{H}^2$, $H^1$, and $H^2$) in Section~\ref{sec:breakdownEPDiff}. Finally in Section~\ref{sec:conclusion} we discuss several possible avenues to generalize these results to higher-order and fractional order EPDiff equations. In Appendix \ref{greenfunctionproof} we prove the Green function formulas, and in Appendix \ref{sec:app_well} we present the local existence theory for homogeneous metrics in $C^1$ using the particle trajectory approach.

\subsection{Acknowledgements}
{The second author is grateful for the hospitality of the University of Vienna, the Wolfgang Pauli Institute, and the Erwin Schr\"odinger Institute. MB was partially supported
by NSF grants DMS-1912037 and DMS-1953244 and by FWF grant FWF-P 35813-N. }

\section{Analytic background material}\label{sec:background}
\subsection{The Liouville equation and ODE comparison theorems}
We start by presenting the explicit solution to the %
Liouville equation~\cite{liouville1853equation}, which will be an essential ingredient for the results of the article:

\begin{lemma}[Explicit Solution of the Liouville Equation]\label{lem:Liouville}
For a fixed function $z_0 \in L^1(\mathbb R_{\geq 0})$
consider
the Liouville-type equation
\begin{equation}\label{eq:Liouville}
\frac{d}{dt}  \ln q(t) = F(q(t)), \qquad
F(q)(r):= \int_r^\infty\frac{z_0(s)}{q(s)}\, ds, \qquad t,r\ge 0.
\end{equation}
Then for any $a>0$, the function  $F$ is locally Lipschitz on the space $C(\mathbb{R}_+,(a,\infty))$ of continuous functions bounded below by $a$.
Thus the equation \eqref{eq:Liouville} with boundary condition $q(0,r)=1$ has a unique solution $q$, which is given by
\begin{equation}\label{eq:Liouville_explicit}
q(t,r)=  \left(1+\frac{t}{2}\int_r^{\infty}z_0(s)ds\right)^2,
\end{equation}
defined on a maximal interval of existence $[0,T)$, where $T=\frac{2}{K}$ with
\begin{equation}
  K=\operatorname{max}\left(-\inf_{r\geq 0} \int_r^{\infty}z_0(s)ds,0\right).
\end{equation}
\end{lemma}
\begin{remark}
Taking an additional $r$ derivative of equation~\eqref{eq:Liouville}  one arrives at the classical
version of the Liouville equation
\begin{equation}\label{eq:Liouville_classical}
\frac{\partial^2}{\partial r\partial t}  \ln q(t,r)=-\frac{z_0(r)}{q(t,r)},\quad t,r\geq 0,
\end{equation}
with boundary conditions of the form
\begin{equation}
q(0,r)=1,\quad \lim_{r\to\infty} q(t,r)=1
\end{equation}
Note that these initial/boundary conditions are more specific than the situation studied in Liouville's original paper~\cite{liouville1853equation}, and our unknown function $q(t,r)$ is the reciprocal of the one he used.
\end{remark}

\begin{proof}[Proof of Lemma~\ref{lem:Liouville}]

Local Lipschitz continuity of $F$ is a special case of Proposition \ref{rhopowerslipschitzthm}.
The Lipschitz continuity of $F$ implies that the function $qF(q)$ is also Lipschitz and thus the theorem of Picard-Lindel\"of  implies the existence and uniqueness of solutions to~\eqref{eq:Liouville}---here we rewrote~\eqref{eq:Liouville} in the form
$
\frac{\partial}{\partial t}  q=q F(q)$.

The explicit solution formula can be easily checked to satisfy the differential equation and boundary conditions by direct computation. It can be derived as in Liouville's original paper~\cite{liouville1853equation},  with slight adaptations stemming from the fact that we use different boundary conditions and have $q$ in the denominator on the right side. See also~\cite{sarria2015sign} for a different derivation of this formula. 
\end{proof}

For the radial Hunter-Saxton equation, treated in Section~\ref{sec:radialHS}, the above theorem will lead to an explicit solution formula in Theorem \ref{exactsolnH1thm}, which will allow us to characterize the precise breakdown mechanism for this family of equations. For the main results, EPDiff equations for non-homogeneous and higher order operators, such an explicit solution formula will not be available.  We will, however, be able to compare the solutions of these more complicated equations to the solution of the above Liouville equation, which will then allow us to deduce finite time breakdown. Towards this aim we will need a comparison theorem for ODEs on Banach spaces. The following theorem is a special case of more general comparison theorems (see e.g.,~\cite{deimling1979existenceII, deimling1979existence}), which is tailored to the situation studied in the present article:
\begin{lemma}[Comparison theorem]\label{lem:comparison}
Let $a>0$ and let
\begin{align}
F,\tilde F:\operatorname{C}(\mathbb{R}_+,(a,\infty))\to \operatorname{C}(\mathbb{R}_+,\mathbb{R})
\end{align}
be functions satisfying the monotonicity property
\begin{equation}\label{comparisonassumption}
\tilde{v}(r) \ge v(r)\; \forall r\ge 0 \quad \Longrightarrow \quad\tilde F(\tilde v)(r)\geq F(v)(r)\; \forall r\ge 0
\end{equation}
for all functions $v,\tilde v\in \operatorname{C}(\mathbb{R}_+,(a,\infty))$. Assume in addition that $\tilde F$
is locally Lipschitz continuous when $\operatorname{C}(\mathbb{R}_+,(a,\infty))$ and $\operatorname{C}(\mathbb{R}_+,\mathbb{R})$ are equipped with the supremum norm.

Let $\tilde u:[0,\tilde T)\to \operatorname{C}(\mathbb{R}_+,(a,\infty))$ be the unique solution to $\partial_t \tilde u=\tilde F(\tilde u)$ with initial condition $\tilde u(0)=\tilde u_0$.  Let $ u:[0, T)\to \operatorname{C}(\mathbb{R}_+,(a,\infty))$  be a solution to $\partial_t u= F( u)$ with the same initial conditions, i.e., $ u(0)= u_0=\tilde u_0$. Then
\begin{equation}
\tilde u(t)(r)\geq u(t)(r),\qquad\forall t\in [0,\operatorname{min}(\tilde T, T))\text{ and } r\in \mathbb R_{+}.
\end{equation}
\end{lemma}

\begin{proof}
Fix $u$, and for any $\tilde{u}$ define $w(t,r) = \tilde{u}(t,r) - u(t,r)$; then $w(0,r)=0$ for all $r\ge 0$, and $w$ satisfies the differential equation 
\begin{equation}\label{wdifferenceequation}
\frac{\partial w}{\partial t}(t,r) = \tilde{F}\big(w(t,r)+u(t,r)\big) - F\big(u(t,r)\big).
\end{equation}
Consider for any $w\colon [0,T]\times [0,\infty)\to [a,\infty)$ the Picard map 
$$ \Phi(w)(t,r) = \int_0^t \Big[ \tilde{F}\big(w(\tau,r)+u(\tau,r)\big) - F\big(u(\tau,r)\big) \Big]\,d\tau.$$
This map is bounded for each bounded $w$, and if $w(\tau,r)\ge 0$ for all $\tau\in [0,T]$ and $r\ge 0$, then 
$w(\tau,r)+u(\tau,r)\ge u(\tau,r)$ and thus our assumption \eqref{comparisonassumption} shows that $\Phi(w)(t,r)\ge 0$ for all $t\in [0,T]$ and $r\ge 0$. 

Thus $\Phi$ maps the closed subspace of bounded nonnegative functions on $[0,T]\times [0,\infty)$ to itself. Since 
$\tilde{F}$ is locally Lipschitz, we see that $\Phi$ is also locally Lipschitz, and thus it is a contraction mapping on a neighborhood of zero. 
For short time the fixed point of iteration exists as in the usual Picard-Lindel\"of theorem, and this fixed point must be in the closed subspace 
of bounded nonnegative functions as well. Hence indeed the unique solution $w(t,r)$ of the differential equation \eqref{wdifferenceequation} is 
nonnegative for short time, and iteratively we can conclude that it is nonnegative for all time that the solution exists.
\end{proof}

\subsection{Green functions for Laplace operators acting on radial vector fields}
The analysis of this article will largely build upon the existence of radial solutions to the (higher-dimensional) EPDiff equation, cf. Lemma~\ref{lem:radialsolutionsEPDiff}. This will essentially allow us to reduce all estimates to estimates on  functions of one variable only. In this section we will collect several formulas for the Green functions of the inertia operator $A=(\aconst\operatorname{id}-\Delta)^k$ acting on radial vector fields.%

Therefore let $U = u(r) \, \partial_r$ be a a smooth radial velocity field. The vector Laplacian in radial coordinates acting on $U$ is given by
\begin{equation}\label{eq:vectorLaplace}
\Delta U = \Big( u''(r) + \frac{n-1}{r}\, u'(r) - \frac{n-1}{r^2} \, u(r)\Big) \, \partial_r.
\end{equation}%
{
To invert this operator we will need to impose certain decay conditions for $r$ towards infinity:
for $\power\in\mathbb{R}$, we let $\decayspace_{\power}$ denote the space of functions that decay at infinity like a power $r^{-\power}$, i.e.,
$$ \decayspace_{\power}(\mathbb R_{\geq 0},\mathbb R) = \{ u\colon \mathbb R_{\geq 0} \to \mathbb{R} \, \vert \, \limsup_{r\to\infty} r^{\power} \lvert f(r)\rvert < \infty\}.$$
In addition to this decay conditon we will assume that the vector fields $u(r)\,\partial_r$ can be extended to smooth vector fields on $\mathbb{R}^n$, which will lead to the condition that all even derivatives at zero vanish. Thus we let 
\begin{equation*}
W^{m,1}_{\text{odd}}(\mathbb R_{\geq 0},\mathbb R):=\left\{u\in W^{m,1}(\mathbb R,\mathbb R)\, \vert \, u^{(2k)}(0)=0,\; 0\leq 2k< m \rfloor\right\},
\end{equation*}
where $W^{m,1}$ denotes the Sobolev space of order $m$, i.e., all functions on $\mathbb{R}$ that have $(m-1)$ absolutely continuous derivatives (so that the $m^{\text{th}}$ derivative exists almost everywhere and is locally in $L^1$). Using these definitions we consider the intersection of those two spaces as  domain for the radial Laplace operator, i.e., we consider the space
\begin{equation}\label{Qspacedef}
\mathcal{Q}^m_{\power}(\mathbb R_{\geq 0},\mathbb R) = W^{m,1}_{\text{odd}}(\mathbb R_{\geq 0},\mathbb R)\cap \decayspace_{\power}(\mathbb R_{\geq 0},\mathbb R).
\end{equation}
We obtain the following result concerning the corresponding Green's functions:}
\begin{lemma}\label{lem:greensfunctionhomog}
{Let $k\in \mathbb N$, and $2(k-1)<\power<n$. Given any $\omega\in \mathcal{Q}^0_{2k+\power-1}(\mathbb R_{\geq 0},\mathbb R)$, there is a unique solution $u\in \mathcal{Q}^{2k}_{\power-1}(\mathbb R_{\geq 0},\mathbb R)$ of 
$(-\Delta)^k\big(u(r)\partial_r\big) = \omega(r)\,\partial_r$. }

{
It is given by
\begin{equation}\label{generalwsoln}
u(r) = \int_0^r \delta(s,r) s^{n-1} \omega(s) \,ds + \int_r^{\infty} \delta(r,s) s^{n-1} \omega(s) \,ds,\qquad \delta(r,s) := rs \varphi(r,s), 
\end{equation}
where the function $\varphi(r,s)$ is defined on $D = \{(r,s) \, \vert s\ge r\ge 0\} \backslash \{(0,0)\}$. For $k=1$ and $k=2$ it is given explicitly via  
\begin{itemize}
\item $k=1$ and $n\geq 1:$ \begin{equation}\label{H1dotinversion}
\varphi(r,s) = \tfrac{1}{n} s^{-n}.
\end{equation}
\item $k=2$ and $n\ge 3:$
\begin{equation}\label{H2dotinversion}
\varphi(r,s) = \tfrac{1}{2n(n-2)} s^{2-n} - \tfrac{1}{2n(n+2)} r^2 s^{-n}.
\end{equation}
\end{itemize}}
\end{lemma}
{
In the non-homogeneous situation, i.e., for $A=(1-\Delta)^k$, the situation is significantly easier, as we do not have to keep track of the decay conditions for $r\to \infty$. To further simplify the situation, we will restrict ourselves to the smooth category and consider the space of odd $H^{\infty}$-functions:
\begin{equation*}
H^{\infty}_{\text{odd}}(\mathbb R_{\geq 0},\mathbb R):=\left\{u\in H^{\infty}(\mathbb R_{\geq 0},\mathbb R)\, \vert \, u^{(2k)}(0)=0,\; \forall k\in\mathbb N\right\};,
\end{equation*}
In this setting we obtain the following result concerning the Green's function for $(1-\Delta)^k$:}
\begin{lemma}\label{lem:greensfunctionnonhomog}
{
For any $\omega\in H^{\infty}_{\text{odd}}(\mathbb R_{\geq 0},\mathbb R)$ and positive integer $k$, there is a unique $u\in H^{\infty}_{\text{odd}}(\mathbb R_{\geq 0},\mathbb R)$ such that 
$$ (1-\Delta)^k(u(r)\,\partial_r) = \omega(r)\,\partial_r.$$
This solution can be written in the same form 
\eqref{generalwsoln} as in Lemma~\ref{lem:greensfunctionhomog}, where now the function $\varphi$ is given on $D$ by
\begin{itemize}
\item $k=1$ and $n\geq 1:$
\begin{equation}\label{H1fullinversion}
\varphi(r,s) = \funca_n(r) \funcb_n(s).
\end{equation}
\item $k=2$ and $n\geq 3:$
\begin{equation}\label{H2fullinversion}
\varphi(r,s) = \tfrac{1}{2}\big[ n\funca_n(r) \funcb_n(s) + \tfrac{1}{n} \funca_n(r) \funcb_{n-2}(s) - n\funca_{n-2}(r) \funcb_n(s)\big].
\end{equation}
\end{itemize}
Here the functions $\funca_p$ and $\funcb_p$ are defined for $p\in\mathbb{R}$ by
\begin{equation}\label{besseldef}
\funca_p(r) := \constyp r^{-p/2} I_{p/2}(r), \qquad \funcb_p(r) := \constyp^{-1} r^{-p/2} K_{p/2}(r),\qquad \constyp := 2^{p/2} \Gamma(\tfrac{p}{2}+1),
\end{equation} }
\end{lemma}
We remind the reader that $\Delta$ is the vector Laplacian; the formulas would be different if we were using the Laplacian on functions. We give the proofs of both Lemmas \ref{lem:greensfunctionhomog} and \ref{lem:greensfunctionnonhomog} in Appendix \ref{greenfunctionproof}.

\begin{remark}[Homogeneous metrics]\label{rem:hom_metrics}
 {Note that the operator $(1-\Delta)$ is invertible as an operator $H^{k+2}(\mathbb{R}^n)$ to $H^{k}(\mathbb{R}^n)$. Furthermore, the operator $(1-\Delta)^{-1}$ makes sense and maps the space of rapidly decaying smooth functions to itself, and can be iterated any number of times. Thus there is no issue with $k=2$ and $n=1,2$ or even any larger integer $k$. On the other hand, as we have seen in Lemma~\ref{lem:greensfunctionhomog}, the homogeneous operator $(-\Delta)$ behaves quite differently: even for a rapidly decaying, radial, smooth vector field $\omega(r)\,\partial_r$, it is not true that the solution of $\Delta(u(r)\,\partial_r) = -\omega(r)\,\partial_r$ is also rapidly decaying. In the above formulas we take advantage of the fact that the slow decay of functions in $\mathbb{R}^n$ for $n\ge 3$ is still fast enough to invert the Laplacian twice, and thus work only with $n\ge 3$ when dealing with the $\dot{H}^2$ metric. See Appendix \ref{greenfunctionproof} for the computations which justify this statement. }

{
 Furthermore, it is important to note that, although we require a power-law decay condition in order to justify that $u$ is the unique solution of $(-\Delta)^k(u\partial_r) = \omega\partial_r$, once we obtain the formula \eqref{generalwsoln}, it makes sense for $\omega$ under milder conditions, in particular 
\begin{equation}\label{minimalomegaintegral}
\int_0^{\infty} r^{2k} \lvert \omega(r)\rvert \,dr < \infty,
\end{equation}
as can be shown iteratively. We will use this fact in Appendix \ref{sec:app_well} to establish a local existence result with minimal assumptions on the initial data $\omega_0$, which includes for example delta functions.}
\end{remark}

\section{The general EPDiff equation %
}\label{sec:EulerArnold}
\subsection{Euler-Arnold equation}\label{sec:EulerArnoldgeneral}
In this part we will recall the derivation of the EPDiff equation as an Euler-Arnold equation on the diffeomorphism group. We will follow the presentation in~\cite{bauer2015local}; see also~\cite{misiolek2010fredholm,vizman2008geodesic} and the references therein.

We will focus the presentation on the non-homogeneous inertia operator, i.e., $ A=\A$ with $\aconst>0$, and only comment on the homogeneous case ($\aconst=0$) at the end of this section, cf. Remark~\ref{rem:hom_metrics}. To derive the EPDiff equation as a geodesic equation we consider the group of $H^{\infty}$-diffeomorphisms on $\R^n$:
\begin{equation*}
\Diff(\R^n):=\left\{\eta=\operatorname{id}+f:f\in H^{\infty}(\R^n,\R^n)\text{ and } \operatorname{det}(\operatorname{id}+df)>0 \right\},
\end{equation*}
where
\begin{equation}
    H^{\infty}(\R^n,\R^n)=\bigcap_{q\geq 0}H^{q}(\R^n,\R^n)
\end{equation}
is the intersection of all Sobolev spaces. The space $\Diff(\R^n)$ is a regular Fr\'echet Lie group with Lie algebra the set of all $H^{\infty}$ vector fields, i.e., $T_e\Diff(\R^n)=\mathfrak X_{H^{\infty}}(\R^n)=H^{\infty}(\R^n,\R^n)$, see~\cite{djebali2009existence}. To define a right invariant metric we have to describe an inner product on
$H^{\infty}(\R^n,\R^n)$. The simplest inner product on this space is the $L^2$-product given by
\begin{equation*}
\langle U_1,U_2 \rangle_{L^2}=\int_{\R^n} U_1\cdot U_2\; dx,
\end{equation*}
where $U_1,U_2\in H^{\infty}(\R^n,\R^n)$ and where $\cdot$ denotes the Euclidean scalar product on $\R^n$. Using the operator $A$ we can define a Sobolev $H^k$ inner product via
\begin{equation*}
\langle U_1,U_2 \rangle_{H^k}=\int_{\R^n} A U_1\cdot U_2\; dx=
\int_{\R^n} \A U_1\cdot U_2\; dx\;,
\end{equation*}
where the symmetry of this inner product follows from the corresponding property of the Laplacian. Given an inner product on the Lie algebra, we can extend this to a right-invariant metric on the diffeomorphism group via right translation, i.e.,
\begin{equation}\label{eq:Hkmetric}
g^k_{\eta}(\delta\eta_1,\delta\eta_2)=  \langle \delta\eta_1\circ\eta^{-1},\delta\eta_2\circ\eta^{-1} \rangle_{H^k}, \qquad \delta\eta_1,\delta\eta_2\in T_{\eta}\Diff(\R^n)
\end{equation}
Using this Riemannian metric we can define the ``kinetic energy'' of a path of diffeomorphisms $\eta:[0,1]\to\Diff(\R^n)$ via
\begin{equation*}
E(\eta) = \frac12\int_0^1 g^k_{\eta}(\dot \eta,\dot \eta) dt.
\end{equation*}
Curves that minimize the energy between fixed endpoints
are called
``geodesic curves,'' and the first order optimality condition
$dE=0$ gives rise to the geodesic equation. For a right-invariant metric on a Lie group, it is convenient to introduce the Eulerian velocity $U(t,x)=\partial_t\eta(t,\eta^{-1}(t,x))$.
Using this change of coordinates the geodesic equation can be written as
\begin{equation}\label{generalgeodesic}
  \partial_t\eta(t,x) = U(t,\eta(t,x)), \qquad U_t(t,x) + \ad(U(t,x))^{\top}U(t,x) = 0,
\end{equation}
where
$\ad(U)^{\top}$ is the adjoint of the operator $\ad(U)$, with respect to the inner product $\langle\cdot,\cdot\rangle_{H^k}$. This form of the geodesic equations on Lie groups was originally derived in the finite-dimensional setting and was extended
by Arnold~\cite{arnold2014differential} to diffeomorphism groups (and more generally to Fr\'{e}chet-Lie groups). For this reason
the first order equation on $U$ is called the Euler-Arnold equation. %
Following this approach for the $H^k$-inner product on $\Diff(\R^n)$, one obtains exactly the EPDiff equation~\eqref{eq:EPDiff} with inertia operator $A=\A$; see e.g.,~\cite{misiolek2010fredholm,bauer2015local}.
As a consequence of this geometric interpretation we obtain a conserved quantity, the Riemannian energy:
\begin{corollary}
Let $U:[0,T)\to C^{\infty}(\R^n,\R^n)$ be a solution to the EPDiff equation~\eqref{eq:EPDiff} with inertia operator $\A$. Then the instantaneous Riemannian energy is constant over time, i.e.,
\begin{equation}
\langle U(t),U(t) \rangle_{H^k}=\operatorname{const}.
\end{equation}
\end{corollary}
The general geodesic equation \eqref{generalgeodesic} implies the momentum conservation law 
\begin{equation*}
\frac{d}{dt} \big(\Ad_{\eta(t)}^{\top} U(t)\big) = 0,
\end{equation*}
relating the flow $\eta(t)$ to the velocity field $U(t)$.
This can be integrated using $\eta(0)=\id$ and $U(0)=U_0$ to obtain 
\begin{equation}\label{vorticityconservation}
\Ad_{\eta(t)}^{\top} U(t) = U_0.
\end{equation}
We may then eliminate $U(t)$ in \eqref{generalgeodesic} to get an equation directly on the group given by 
\begin{equation}\label{groupODE}
\frac{d\eta}{dt} = U(t)\circ\eta(t) = \Ad_{\eta(t)^{-1}}^{\top} U_0\circ\eta(t), \qquad \eta(0)=\id.
\end{equation}
Explicitly, on the diffeomorphism group with inertia operator $A=\A$, this takes the form (see \cite{misiolek2010fredholm}, Proposition 3.6)
$$ \Ad_{\eta}^{\top} U_0 = A^{-1}\big[ \Jac(\eta) D\eta^T (AU_0)\big].$$
If $A$ is a sufficiently strong differential operator---$k\ge 1$ is sufficient---then \eqref{groupODE} has a smooth right side as a function of $\eta$ in the Sobolev space $H^s$ for large $s$, and we can prove local well-posedness using Picard iteration, for any fixed $U_0\in H^s$. 
When applied to the Euler equations (on the volume-preserving diffeomorphism group), this technique is referred to as 
the \emph{particle-trajectory method} in Majda-Bertozzi~\cite{majda_bertozzi_2001}, see also ~\cite{ebin1984concise}.  A derivation on abstract Lie groups can be found in~\cite[Lemma~5]{bauer2016geometric}.  

We will apply this method to not only study local well-posedness but also global existence, in the special case that 
$U_0$ is a purely radial vector field $U_0 = u_0(r) \, \partial_r$, and thus $\eta(t)$ is a purely radial diffeomorphism,
with $\eta(t,x) = \gamma(t,r) x/r$ for $x\in \mathbb{R}^n$ and $r=\lvert x\rvert$.

\subsection{Local ODE existence and radial solutions}
As a consequence of the geometric description in the previous section, a series of local and global well-posedness results for the EPDiff equation (depending on the order $k$) have been obtained~\cite{ebin1970groups,misiolek2010fredholm,gay2009well,mumford2013euler}. In the following theorem we summarize these results for initial data in $H^{\infty}(\R^n,\R^n)$:
\begin{proposition}[Local and global well-posedness]\label{prop:wellposedness}
Let $A=(1-\Delta)^k$.  For any $k\geq 1$ and
$U_{0}\in H^{\infty}(\R^n,\R^n)$, the EPDiff equation~\eqref{eq:EPDiff} has a unique non-extendable smooth solution
\begin{equation*}
  U\in C^\infty(J,H^{\infty}(\R^n,\R^n)).
\end{equation*}
The maximal interval of existence $J$ is open and contains $0$.
If $k>\frac{n}{2}+1$, then the solution exists for all time $t$, i.e., $J=\R$.
\end{proposition}
\begin{remark}[Fractional orders]
In~\cite{escher2014right,bauer2015local} the above result has been extended to fractional order EPDiff equations, i.e., to $k\in \R$, where it has been shown that  the local well-posedness result holds for any $k\geq \frac{1}{2}$ and the global well-posedness result for $k>\frac{n}{2}+1$.
\end{remark}
\begin{remark}[Setting for wellposedness results]
By the above result, Proposition~\ref{prop:wellposedness},
we have local existence for  the nonhomogeneous $H^k$ metric and $\omega_0$ in $H^{\infty}$. For the homogeneous $\dot{H}^k$ case  and $\omega_0\in L^1$ (and thus $\gamma$ being barely $C^1$) we show local well-posedness in Appendix~\ref{sec:app_well}.  With additional work, we could prove the opposite---local existence in $H^{\infty}$ for homogeneous $\dot{H}^k$ metrics, and local existence in $C^1$ for nonhomogeneous $H^k$ metrics---but this would take us too far afield. In either case a solution $\gamma$ must be at least $C^1$ spatially to be considered a classical solution, and so our result that the solution must fail to be globally $C^1$ applies regardless of how a local solution is obtained or what the initial data is.
\end{remark}

In the next result, which will be essential in the remainder of the article, we observe the existence of radial solutions to the EPDiff equation.
\begin{lemma}[Radial Solutions]\label{lem:radialsolutionsEPDiff}
Let $U_0$ be a radial initial velocity, i.e., $U_0=u_0(r)\partial_r$. Let $U$ be the corresponding solution to the EPDiff equation~\eqref{eq:EPDiff} defined on its maximal interval of existence $J$. Then $U$ is a radial velocity field for each $t\in J$. The radial function $u$ satisfies the corresponding radial EPDiff equation given by \eqref{mainomega}.
\end{lemma}

\begin{proof}
This follows directly by plugging the ansatz  $U=u(t,r)\partial_r$ into~\eqref{eq:EPDiff}, using the fact that
$\A$ preserves radial velocity fields.
\end{proof}

The momentum form of the Euler-Arnold equation given by \eqref{mainomega} implies a transport law which will be crucial for the rest of the paper. To derive it we define the radial flow map $\gamma(t,r)$ by the flow equation 
\begin{equation}\label{radialflow}
\frac{\partial \gamma}{\partial t}(t,r) = u\big(t, \gamma(t,r)\big), \qquad \gamma(0,r)=r,
\end{equation}
which is just \eqref{generalgeodesic} specialized to the radial case.

\begin{lemma}\label{lem:particletrac}
If $U(t,r) = u(t,r)\,\partial_r$ and $\omega(t,r)$ solve  the radial EPDiff equation \eqref{mainomega} on a time interval $[0,T)$ for all $r\ge 0$, with the flow $\gamma(t,r)$ defined by \eqref{radialflow}, then for all $t\in [0,T)$ and $r\in [0,\infty)$ we have the transport law
\begin{equation}\label{vorticitytransportm}
\gamma(t,r)^{n-1} \gamma_r(t,r)^2 \omega\big(t, \gamma(t,r)\big) = r^{n-1} \omega_0(r).
\end{equation}
\end{lemma}

\begin{proof}
By the chain rule and product rule, we have 
\begin{align*}
&\frac{\partial}{\partial t}\ln \Big( \gamma(t,r)^{n-1} \gamma_r(t,r)^2  \omega\big(t, \gamma(t,r)\big)\Big) \\
&\qquad\qquad= \frac{\omega_t\big(t,\gamma(t,r)\big) + \gamma_t(t,r) \omega_r\big(t,\gamma(t,r)\big)}{\omega\big(t,\gamma(t,r)\big)} 
+ (n-1)\,\frac{\gamma_t(t,r)}{\gamma(t,r)} + 2 \, \frac{\gamma_{tr}(t,r)}{\gamma_r(t,r)} \\
&\qquad\qquad= \frac{\omega_t\big(t,\gamma(t,r)\big) + u\big(t,\gamma(t,r)\big) \omega_r\big(t,\gamma(t,r)\big)}{\omega\big(t,\gamma(t,r)\big)} 
+ (n-1)\,\frac{u\big(t,\gamma(t,r)\big)}{\gamma(t,r)} + 2 \, u_r\big(t,\gamma(t,r)\big),
\end{align*}
using the fact that $\gamma_t(t,r) = u\big(t,\gamma(t,r)\big)$ and the $r$-derivative is 
\begin{equation}\label{gammarderivative}
\gamma_{tr}(t,r) = u_r\big(t,\gamma(t,r)\big) \gamma_r(t,r).
\end{equation}
Now replace $\gamma(t,r)$ everywhere with a new variable $R$, and this becomes 
\begin{align*}
&\frac{\partial}{\partial t}\ln \Big( \gamma(t,r)^{n-1} \gamma_r(t,r)^2  \omega\big(t, \gamma(t,r)\big)\Big) \\
&\qquad\qquad= \frac{1}{\omega(t,R)} \Big[ \omega_t(t,R) + u(t,R) \omega_r(t,R)
+ (n-1)\,\frac{u(t,R)}{R} \,\omega(t,R)  + 2 \, u_r(t,R) \omega(t,R)\Big],
\end{align*}
and the term in brackets is zero since it is just \eqref{mainomega} evaluated at $R$. 

Hence the quantity $\gamma(t,r)^{n-1} \gamma_r(t,r)^2  \omega\big(t, \gamma(t,r)\big)$ is constant in time. Since $\gamma(0,r)=r$ and $\gamma_r(0,r)=1$, the value of this quantity at $t=0$ is $r^{n-1} \omega_0(r)$. 
\end{proof}

Finally we discuss the main breakdown result we derive in this paper. A classical solution $U(t,x)$ exists as long as $U$ remains $C^1$ in the spatial variable $x\in \mathbb{R}^n$, corresponding to $u(t,r)$ being $C^1$ in $r\ge 0$. We will prove breakdown by showing that the radial Lagrangian flow $\gamma$ must leave the diffeomorphism group in finite time, which happens if and only if $u_r(t,r)$ becomes unbounded in finite time.

\begin{lemma}\label{breakdownlemma}
{Let $u\colon [0,T)\times [0,\infty)\to\mathbb{R}$ be a $C^1$ solution of the radial EPDiff equation \eqref{mainomega} such that $\lim_{r\to\infty} u(r)=0$.
We have:
\begin{enumerate}
\item The Lagrangian flow, defined by $\gamma_t(t,r) = u(t,\gamma(t,r))$ with $\gamma(0,r)=r$, exists on the same time interval as $u$ and is $C^1$ in time and space.
\item If $\lim_{t\nearrow T}\gamma_r(t,r)=0$ for some $r\ge 0$ and $T>0$, then 
$u$ cannot be extended as a $C^1$ solution to time $T$.
\end{enumerate}
} 
\end{lemma}

\begin{proof}
{
The Lagrangian flow is an ODE on $\mathbb{R}^n$ for a $C^1$ time-dependent vector field which decays as $r\to \infty$ and is thus bounded.  Consequently the flow is also defined on the same time interval and is $C^1$ in time and space, see~\cite[Proposition 4.1.22]{abraham2012manifolds}.}

  To see the second statement we  integrate equation \eqref{gammarderivative} and use the initial condition $\gamma_r(t,0)=1$. Then we have that 
    $$ \ln{\gamma_r(t,r)} = \int_0^t  u_r\big(\tau,\gamma(\tau,r))\big)\,d\tau$$
    for each $r\ge 0$.
    If $\gamma_r(t,r)$ approaches zero as $t\to T$ for some $r\ge 0$, then both sides of this equation must approach negative infinity for that $r$, and in particular $\displaystyle \inf_{r\ge 0} u_r(t,r) \to -\infty$ as $t\to T$. So the $C^1$ norm of $u$ must approach infinity, and the solution cannot be continued classically.
\end{proof}

\subsection{Momentum transport formulation}\label{sec:vort_trans}
Finally we will describe the momentum transport formulation, as discussed in general at the end of Section~\ref{sec:EulerArnold}, for the specific case of radial solutions. %
\begin{proposition}\label{gammadeltaformprop}
Suppose that the inertia operator $A$ is invertible and that the solution of $A(u(r)\,\partial_r) = \omega(r)\,\partial_r$ is given by an integral formula 
of the form \eqref{generalwsoln},
where $\delta$ is $C^1$ on $D = \{(r,s) \, \vert \, s\ge r\ge 0\} \backslash \{(0,0)\}$. 
Let $u(t,r)$ be a solution of the radial EPDiff equation \eqref{mainomega} with $u(0,r)=u_0(r)$ and $\omega_0(r)\,\partial_r = A(u_0(r)\partial_r)$. If $z_0(r):=r^{n-1}\omega_0(r)$, then the flow $\gamma(t,r)$ satisfies 
\begin{equation}\label{gammadeltaform}
\frac{\partial \gamma}{\partial t}(t,r) = \int_0^r \frac{\delta\big( \gamma(t,s),\gamma(t,r)\big)}{\gamma_s(t,s)}  \, z_0(s)\,ds + 
\int_r^{\infty} \frac{\delta\big( \gamma(t,r),\gamma(t,s)\big)}{\gamma_s(t,s)}  \, z_0(s)\,ds, 
\end{equation}
and its spatial derivative satisfies 
\begin{equation}\label{gammaprimedeltaform}
\frac{\partial}{\partial t}\ln{\big(\gamma_r(t,r)\big)} = \int_0^r \frac{\partial_2\delta\big( \gamma(t,s),\gamma(t,r)\big)}{\gamma_s(t,s)}  \, z_0(s)\,ds + 
\int_r^{\infty} \frac{\partial_1\delta\big( \gamma(t,r),\gamma(t,s)\big)}{\gamma_s(t,s)}  \, z_0(s)\,ds, 
\end{equation}
\end{proposition}

\begin{proof}
We want to express 
$\gamma_t(t,r) = u\big(t,\gamma(t,r)\big)$ as an integral formula using the conservation law \eqref{vorticitytransportm}. 
Fix a $t$; since the inversion of $A$ involves only spatial computations, we can ignore the $t$-dependence. 
So we have 
$$  u\big(\gamma(r)\big) = 
\int_0^{\gamma(r)} \delta\big(\sigma,\gamma(r)\big) \sigma^{n-1} \omega(\sigma) \,d\sigma + \int_{\gamma(r)}^{\infty} \delta\big(\gamma(r),\sigma\big) \sigma^{n-1}\omega(\sigma)\,d\sigma.$$
Now change variables via $\sigma = \gamma(s)$, so that $d\sigma = \gamma'(s) \,ds$. Then we have 
$$ u\big(\gamma(r)\big) = 
\int_0^r \delta\big(\gamma(s),\gamma(r)\big) \gamma(s)^{n-1} \omega\big(\gamma(s)\big) \gamma'(s) \,ds + \int_r^{\infty} \delta\big(\gamma(r),\gamma(s)\big) \gamma(s)^{n-1}\omega\big(\gamma(s)\big) \gamma'(s) \,ds.$$
Formula \eqref{vorticitytransportm} then implies that $\gamma(r)^{n-1}\gamma'(r)^2 \omega\big(\gamma(r)\big) = z_0(r)$, so these integrals simplify to 
$$ u\big(\gamma(r)\big) = 
\int_0^r \delta\big(\gamma(s),\gamma(r)\big) \, \frac{z_0(s)}{\gamma'(s)} \,ds + \int_r^{\infty} \delta\big(\gamma(r),\gamma(s)\big) \,\frac{z_0(s)}{\gamma'(s)} \,ds,$$
which is \eqref{gammadeltaform}.

To compute the spatial derivative, we use the Leibniz integral rule (relying on the fact that $\delta$ is $C^1$) and obtain 
\begin{align*}
u'\big(\gamma(r)\big) \gamma'(r) &= 
\int_0^r \partial_2\delta\big(\gamma(s),\gamma(r)\big)\,\gamma'(r) \, \frac{z_0(s)}{\gamma'(s)} \,ds + \int_r^{\infty} \partial_1\delta\big(\gamma(r),\gamma(s)\big)\,\gamma'(r) \,\frac{z_0(s)}{\gamma'(s)} \,ds \\
&\qquad\qquad +\delta\big(\gamma(r),\gamma(r)\big) \, \frac{z_0(r)}{\gamma'(r)} - \delta\big(\gamma(r),\gamma(r)\big) \,\frac{z_0(r)}{\gamma'(r)}.
\end{align*}
The local (nonintegral) terms on the last line cancel out, and in what remains there is a common factor of $\gamma'(r)$ which can be canceled to yield
\begin{align*} 
\frac{\partial}{\partial t} \ln{\big(\gamma_r(t,r)\big)} &= \frac{\gamma_{tr}(t,r)}{\gamma_r(t,r)} = u'\big(\gamma(r)\big)  \\
&= 
\int_0^r \partial_2\delta\big(\gamma(s),\gamma(r)\big) \, \frac{z_0(s)}{\gamma'(s)} \,ds + \int_r^{\infty} \partial_1\delta\big(\gamma(r),\gamma(s)\big)  \,\frac{z_0(s)}{\gamma'(s)} \,ds,
\end{align*}
which is \eqref{gammaprimedeltaform}.
\end{proof}

The system \eqref{gammadeltaform}--\eqref{gammaprimedeltaform} can be written in the form of an autonomous vector field ODE on a Banach space, as 
\begin{align}
\frac{d\gamma}{dt}(r) &= \int_0^r \frac{\delta\big( \gamma(s),\gamma(r)\big)}{\rho(s)}  \, z_0(s)\,ds + 
\int_r^{\infty} \frac{\delta\big( \gamma(r),\gamma(s)\big)}{\rho(s)}  \, z_0(s)\,ds, \label{gammaODE} \\
\frac{d\rho}{dt}(r) &= \rho(r) \int_0^r \frac{\partial_2\delta\big( \gamma(s),\gamma(r)\big)}{\rho(s)}  \, z_0(s)\,ds + 
\rho(r) \int_r^{\infty} \frac{\partial_1\delta\big( \gamma(r),\gamma(s)\big)}{\rho(s)}  \, z_0(s)\,ds. \label{rhoODE}
\end{align}
Here $\rho(t,r)=\gamma_r(t,r)$, but we treat it as a separate variable to get a closed ODE system.

Of course $\gamma$ and $\rho$  are not independent, since $\gamma$ is the unique antiderivative of $\rho$ such that $\gamma(0)=0$. Hence we may consider the system as a single ODE for $\rho$ alone. Hence we may view \eqref{gammaODE} as a consequence of \eqref{rhoODE}, although we will often find it convenient to deal with both equations simultaneously. For a fixed function $z_0$, this makes sense on the space of functions $\rho$ bounded above and below by positive numbers.  

More precisely, consider $\rho\in \spacey$ where $\spacey$ is the space of continuous functions on $[0,\infty)$ for which there exist numbers $b>a>0$
such that $0<a\le \rho(r)\le b$ for all $r\in [0,\infty)$. Then we also have
 $$ a\le \frac{\gamma(r)}{r} = \frac{\int_0^r \rho(s)\,ds}{r} \le b,$$
so that $r\mapsto\gamma(r)/r$ is also in $\spacey$.  
Let $\Gamma$ denote the map from $\spacey$ to $C^1$ diffeomorphisms given by 
\begin{equation}
\Gamma(\rho)(r) =\gamma(r).
\end{equation} 
Then we can express the vector field $X$ which governs the evolution of $\rho$ via:
\begin{multline}\label{vectorfield}
X(\rho)(r) =  \rho(r) \int_0^r \partial_2\delta\big(\Gamma(\rho)(s),\Gamma(\rho)(r)\big) \, \frac{z_0(s)}{\rho(s)} \, ds
+ \rho(r) \int_r^{\infty} \partial_1\delta\big(\Gamma(\rho)(r),\Gamma(\rho)(s)\big) \, \frac{z_0(s)}{\rho(s)} \, ds.
\end{multline}
Fundamentally this is what we have in mind when considering the system \eqref{gammaODE}--\eqref{rhoODE}.

In Appendix \ref{sec:app_well} we prove local existence of solutions of the system \eqref{gammaODE}--\eqref{rhoODE} in the homogeneous case $\aconst=0$ with $k=1$ and $k=2$. We obtain local solutions $\rho$ in the space of continuous, positive functions on $[0,\infty)$, and thus we obtain existence of $C^1$ local solutions $\gamma$. These then indirectly generate spatially $C^1$ local solutions $u(t,r)$ via the formulas $\partial_t\gamma(t,r) = u(t,\gamma(t,r))$ and $u_r(t,\gamma(t,r)) = \rho_t(t,r)/\rho(t,r)$. 
See~\cite{misiolek2002classical,lee2017local} for a similar approach in the context of the Camassa-Holm equation.

\section{An explicit solution formula for the radial Hunter-Saxton equation}\label{sec:radialHS}
In this section we will use the solution of the Liouville equation to obtain a solution formula for the radial Hunter-Saxton equation with $\dot{H}^1$ inertia operator and thereby see the precise breakdown mechanism. 

The radial Hunter-Saxton equation corresponds to the EPDiff equation for radial solutions with the inertia operator $A=-\Delta$, i.e.,
\begin{equation}\label{eq:radialHS}
\begin{cases}
&\omega_t + u \omega_r + 2u_r \omega + \tfrac{n-1}{r} u \omega = 0,\\
&\omega\partial_r=-\Delta(u\partial_r)=-\Big(u_{rr} + \frac{n-1}{r}\, u_r - \frac{n-1}{r^2} \, u\Big)\partial_r.\
\end{cases}
\end{equation}
Note that in dimension one the momentum is simply $\omega=-u_{rr}$, and equation~\eqref{eq:radialHS} equals  the Hunter-Saxton equation~\cite{hunter1991dynamics}, hence the name radial Hunter-Saxton equation. The higher-dimensional Hunter-Saxton equation was introduced by Modin in~\cite{modin2015generalized}, motivated by connections to the field of information geometry~\cite{khesin2013geometry}.
Recalling formula \eqref{H1dotinversion}, we have that $\delta(r,s) = \tfrac{1}{n} r s^{1-n}$; thus the system \eqref{gammaODE}--\eqref{rhoODE} can be written as 
\begin{equation}\label{geodesicH1system}
\begin{split}
\frac{d\gamma}{dt}(r) =  \frac{\gamma(r)^{1-n}}{n}  \int_0^r \gamma(s) \, \frac{z_0(s)}{ \rho(s)} \,ds
+ \frac{\gamma(r)}{n} \int_r^{\infty} \gamma(s)^{1-n} \frac{z_0(s)}{\rho(s)}\,ds, \\
\frac{d\rho}{dt}(r) = \frac{(1-n) \gamma(r)^{-n} \rho(r)}{n} \int_0^r \gamma(s) \, \frac{z_0(s)}{ \rho(s)} \,ds
+ \frac{\rho(r)}{n} \int_r^{\infty} \gamma(s)^{1-n} \frac{z_0(s)}{\rho(s)}\,ds,
\end{split}
\end{equation}
where $z_0(s) = s^{n-1} \omega_0(s)$. 
In the next theorem we will derive an explicit solution formula for equation~\eqref{geodesicH1system} and thus of the radial Hunter-Saxton equation~\eqref{eq:radialHS}. Our strategy is to combine these equations into one, eliminating the integral on $[0,r]$ so that only the integral on $[r,\infty)$ remains. By coincidence everything in the combined equation depends only on the quantity $\gamma(r)^{n-1}\rho(r)/r^{n-1}$, which then satisfies the Liouville equation \eqref{eq:Liouville}. 
\begin{theorem}\label{exactsolnH1thm}
{Let $\omega_0\in  L^1([0,\infty))$}.
Then
\begin{align}\label{H1soln}
\frac{\partial \gamma}{\partial r}(t,r) &= \rho(t,r)\\
\gamma(t,r)^{n-1} \rho(t,r) %
&=r^{n-1} \Big( 1 + \frac{t}{2} \int_r^{\infty} \omega_0(s)\,ds \Big)^2
= r^{n-1} \left( 1 + \frac{t}{2}\left(u_0'(r)+\frac{n-1}{r}u_0(r)\right) \right)^2
\end{align}
solves the Lagrangian form of the radial Hunter-Saxton equation~\eqref{geodesicH1system} with initial data $\gamma(0,r)=r$ and $\frac{\partial \gamma}{\partial r}(0,r) = \rho(0,r)=1$.

The solution breaks down in finite time by leaving the diffeomorphism group, if and only if $u_0'(r)+\frac{n-1}{r}u_0(r) < 0$ for some $r\ge 0$.
\end{theorem}

\begin{proof}
The system \eqref{geodesicH1system} can be written in the form
\begin{align*}
\frac{\partial}{\partial t} \ln{\gamma} &= \frac{1}{n} \big[ X(t,r) + Y(t,r)\big], \\
\frac{\partial}{\partial t} \ln{\rho} &= \frac{1}{n} \big[ -(n-1) X(t,r) + Y(t,r)\big],
\end{align*}
where
$$ X(t,r) := \gamma(t,r)^{-n} \int_0^r \frac{\gamma(t,s) s^{n-1}\omega_0(s)}{ \rho(t,s)} \,ds \qquad \text{and}\qquad
Y(t,r) :=  \int_r^{\infty} \frac{s^{n-1}\omega_0(s)}{\gamma(t,s)^{n-1}\rho(t,s)}\,ds.$$
Thus we see that
\begin{equation}\label{H1logform}
\frac{\partial}{\partial t}\Big( \ln{\big[ \gamma(t,r)^{n-1} \rho(t,r)\big]}\Big) = Y(t,r) %
\end{equation}
 Define a new function
 \begin{equation}\label{qdef}
 q(t,r) := \frac{\gamma(t,r)^{n-1} \rho(t,r)}{r^{n-1}} \quad \text{for $r>0$}, \qquad q(t,0) := \rho(t,0)^{n},
 \end{equation}
 where the value at $r=0$ is defined to get continuity, and the quotient by $r^{n-1}$ is to ensure that if $\rho(t,r)$ is bounded between two positive constants $a$ and $b$ for some $t$, then $q(t,r)$ will be bounded between $a^n$ and $b^n$ (and hence be an element of our space of continuous positive functions; see Appendix \ref{sec:app_well}).

 Then equation \eqref{H1logform} becomes
\begin{equation}\label{Liouville1}
\frac{\partial}{\partial t} \ln{q(t,r)} = \int_r^{\infty} \frac{\omega_0(s)}{q(t,s)} \,ds, \qquad q(t,0) = 1,
\end{equation}
 which puts it exactly in the form of the Liouville equation studied in  Lemma~\ref{lem:Liouville}. Thus it follows that
 \begin{equation*}%
q(t,r) = r^{n-1} \Big( 1 + \frac{t}{2} \int_r^{\infty} \omega_0(s)\,ds \Big)^2
\end{equation*}
From here the solution formula follows by integrating. 
\begin{equation}
\int_r^{\infty} \omega_0(s)\,ds=-\int_r^{\infty} \Big[ u_0''(s)+\frac{n-1}{s}u_0'(s)-\frac{n-1}{s^2}u_0(s)\Big]   \,ds=
u_0'(r)+\frac{n-1}{r}u_0(r).
\end{equation}
{Here we used that $\lim_{r\to \infty}u_0(r)=\lim_{r\to \infty}u'_0(r)=0$, which follows from the assumption that $\omega_0\in L^1([0,\infty))$ using formulas~\eqref{generalwsoln} and~\eqref{H1dotinversion}.}
Thus the solution can only break down as $\rho(t,r)$ approaches zero in finite time, and this happens if there exists an $r_0\ge 0$ such that $u_0'(r_0)+\frac{n-1}{r_0}u_0(r_0) < 0$.
\end{proof}

 Note that the first such time $t$ with $\rho(t,r_0)=0$ occurs at a point $r_0$ where $\int_r^{\infty} \omega_0(s)\,ds$ has its minimum, so that the initial momentum $\omega_0(r)$ crosses from positive to negative at $r_0$ as in the one-dimensional case. Also note that we could if desired integrate \eqref{H1soln} in $r$ to solve for $\gamma(t,r)^n$, then use this to construct the velocity field $u(t,r)$ using the flow equation 
 \eqref{radialflow}.

\section{The general breakdown result}\label{sec:general}

We now consider the system \eqref{gammaODE}--\eqref{rhoODE} in general. 
The following general breakdown theorem  provides one of the main technical contributions of the present article:
\begin{theorem}\label{Qtheorem}
Let $\omega_0(r)\leq 0$ for all $r>0$ and let $\delta$ be a kernel of the form $\delta(r,s) = rs \varphi(r,s)$, where $\varphi$ is smooth on 
$D = \{(r,s)\, \vert s\ge r\ge 0\} \backslash \{(0,0)\}$, satisfying the following conditions:
\begin{enumerate}[label=(\alph*)]
\item positivity: $\varphi(r,s) > 0$ for all $(r,s)\in D$; \label{main_cond1}
\item log-supermodularity: $\frac{\partial^2}{\partial r\partial s} \ln{\varphi(r,s)} \ge 0$ for all $(r,s)\in D; $\label{main_cond4}
 \item there exists a $C>0$ such that for all $r\ge 0$, \label{main_cond5}
\begin{equation}\label{Sdef}
S(r) := \frac{r \partial_1\varphi(r,r) \varphi(0,r)
- r \varphi(r,r) \partial_2\varphi(0,r)}{\varphi(0,r)^2}\ge C
\end{equation}
\end{enumerate}
Assume in addition that there exists a local solution to equations~\eqref{gammaODE}--\eqref{rhoODE} corresponding to  initial condition
$\omega_0$.

If $Q\colon [0,\infty)\to [0,\infty)$ is defined via
\begin{equation}\label{Qdef}
Q(r) :=\frac{1}{r \varphi(0,r)} \text{ for $r>0$},\qquad Q(0) := \lim_{r\to 0} \frac{1}{r\varphi(0,r)},
\end{equation}
then for any pair $(\gamma,\rho)$ satisfying equations~\eqref{gammaODE}--\eqref{rhoODE}, the quantity $q(t,r) = Q\big(\gamma(t,r)\big)\rho(t,r)/Q(r)$ satisfies the differential inequality
\begin{equation}\label{vectorfieldcomparison}
\frac{\partial}{\partial t} \ln{\big[q(t,r)\big]} \le -C \int_r^{\infty} \frac{s^{n-1}\lvert \omega_0(s)\rvert/Q(s)}{q(t,s)} \, ds.
\end{equation}
Consequently $Q(\gamma(t,r))\rho(t,r)/Q(r)$ reaches zero in finite time by comparison with the Liouville equation.
\end{theorem}

\begin{remark}
Note that $\delta(r,s)$ and $\varphi(r,s)$ are intentionally only defined for $s\ge r$ as in the integrals for $s\in [r,\infty)$; hence when we use $\delta(s,r)$ or $\varphi(s,r)$, it is implicit that $s\le r$ as in the integrals for $s\in [0,r]$. This is important since we will interchange the names of these variables several times, 
and confusion is avoided by remembering that the second argument is always larger than the first. 

Also note that $\partial_1\varphi(r,r)$ is not the same as $\frac{d}{dr} \varphi(r,r)$; rather it is $\lim_{s\to r^-} \frac{\partial}{\partial r} \varphi(r,s)$. Hence the quantity $S(r)$ in \eqref{Sdef} is \emph{not} the same thing as the tempting simplification $r \, \frac{d}{dr} \frac{\varphi(r,r)}{\varphi(0,r)}$. 

Finally note that existence of $Q$ in the definition \eqref{Qdef} at zero is assumed implicitly, and this limit exists in all cases given by Lemmas \ref{lem:greensfunctionhomog} and \ref{lem:greensfunctionnonhomog}. 
When $n=1$ in the cases $\dot{H}^1$ or $H^1$, and when $n=3$ in the cases $\dot{H}^2$ or $H^2$, this limit is positive, while in all higher dimensional cases we have $Q(0)=0$. 
\end{remark}

\begin{proof}[Proof of Theorem~\ref{Qtheorem}]
In the following we will suppress the dependence on $t$. Using the system \eqref{gammaODE}--\eqref{rhoODE} and the assumption $z_0\le 0$, we derive an ODE for $Q(\gamma)\rho$:
\begin{align*}
\frac{d}{dt} \ln{\big[ Q(\gamma)\rho\big]} &=
 \frac{Q'\big(\gamma(r)\big)}{Q\big(\gamma(r)\big)} \, \frac{d\gamma}{dt} + \frac{d}{dt} \ln{\rho} \\
&= -\int_0^r \Big( \frac{Q'(\gamma(r))}{Q(\gamma(r))} \, \delta\big(\gamma(s),\gamma(r)\big) + \partial_2\delta\big(\gamma(s),\gamma(r)\big)\Big) \, \frac{\lvert z_0(s)\rvert}{\rho(s)} \, ds\\
&\qquad\qquad - \int_r^{\infty}\Big( \frac{Q'(\gamma(r))}{Q(\gamma(r))} \, \delta\big(\gamma(r),\gamma(s)\big) + \partial_1\delta\big(\gamma(r),\gamma(s)\big)\Big)\, \frac{\lvert z_0(s)\rvert}{\rho(s)} \, ds
\end{align*}
We will therefore prove inequality~\eqref{vectorfieldcomparison} if we can show that there exists a constant $C>0$ such that
\begin{align}
\frac{Q'(r)}{Q(r)} \, \delta(s,r) + \partial_2\delta(s,r) &\ge 0 \quad \text{for all $s\in [0,r]$;} \label{condition1} \\
\frac{Q'(r)}{Q(r)} \, \delta(r,s) + \partial_1\delta(r,s) &\ge \frac{C}{Q(s)} \quad \text{for all $s\in [r,\infty)$.} \label{condition2}
\end{align}

By assumption \ref{main_cond4}, using $\delta(r,s) = rs \varphi(r,s)$ we have that
\begin{equation}\label{deltaconvexity}
\frac{\partial^2}{\partial r\partial s} \ln{\delta(r,s)} =\frac{\partial^2}{\partial r\partial s} \ln{\varphi(r,s)}  \geq 0 \quad \text{for $s\ge r$},
\end{equation}
which, after interchanging the variables, is equivalent to $$ \frac{\partial}{\partial s} \Big( \frac{1}{\delta(s,r)}\frac{\partial \delta(s,r)}{\partial r}\Big) \ge 0 \quad \text{for $s\le r$}.$$
Integrating this with respect to $s$, on the interval from $0$ to $s$, we get
$$  \frac{\partial_2\delta(s,r)}{\delta(s,r)} \ge \lim_{s\to 0} \frac{\partial_2\delta(s,r)}{\delta(s,r)} = \frac{\frac{d}{dr}[ r\varphi(0,r)]}{r\varphi(0,r)} = -\frac{Q'(r)}{Q(r)},$$
by definition of $Q$, which is equivalent to \eqref{condition1} since $\delta$ is nonnegative.

Since $\varphi$ is positive by assumption \ref{main_cond1}, condition \eqref{condition2} is equivalent to
\begin{equation}\label{cond2prime}
\frac{Q'(r)}{Q(r)} \ge - \frac{\partial_1\delta(r,s)}{\delta(r,s)} + \frac{C}{Q(s)\delta(r,s)} \qquad \text{for all $s\ge r$.}
\end{equation}
Since $Q$ is nonnegative, condition~\eqref{condition1} implies that
$$ \frac{\partial}{\partial r} \big( Q(r)\delta(s,r)\big) \ge 0 \text{ for $r\ge s$,}$$
which implies that $Q(r)\delta(s,r) \ge Q(s) \delta(s,s)$ whenever $r\ge s$. Exchanging the roles of $r$ and $s$ here, we find
that $Q(s)\delta(r,s) \ge Q(r) \delta(r,r)$ whenever $s\ge r$, which implies that for any $C>0$,
\begin{equation}\label{sup1}
\sup_{s\ge r} \frac{C}{Q(s)\delta(r,s)} = \frac{C}{Q(r)\delta(r,r)}.
\end{equation}
Meanwhile the log-supermodularity assumption \eqref{deltaconvexity} implies that
$ \frac{\partial_1\delta(r,s)}{\delta(r,s)}$ is increasing in $s$ for $s\ge r$, and thus
\begin{equation}\label{sup2}
\sup_{s\ge r} - \frac{\partial_1\delta(r,s)}{\delta(r,s)} = -\inf_{s\ge r} \frac{\partial_1\delta(r,s)}{\delta(r,s)} = -\frac{\partial_1\delta(r,r)}{\delta(r,r)}.
\end{equation}
Since \eqref{sup1} and \eqref{sup2} are attained at the same point $s=r$, we see that \eqref{cond2prime}
is satisfied if and only if
$$ \frac{Q'(r)}{Q(r)} \ge -\frac{\partial_1\delta(r,r)}{\delta(r,r)} + \frac{C}{Q(r)\delta(r,r)} \quad \text{for all $r>0$}.$$
This is now equivalent to
\begin{equation}\label{Sconditionsimple}
Q'(r) \delta(r,r) + Q(r) \partial_1\delta(r,r) \ge C,
\end{equation}
and inserting the definition \eqref{Qdef} of $Q$ turns this into
 the requirement $S(r)\ge C$, with $S$ given by \eqref{Sdef} after substituting in terms of $\varphi$ and simplifying.

To prove breakdown of the solution to equation~\eqref{vectorfield} we can again use the comparison theorem, Lemma~\ref{lem:comparison}. Local Lipschitz continuity of the 
comparison function
\begin{equation}
 \tilde{F}(q)(r):=-C q(r) \int_r^{\infty} \frac{ \lvert z_0(s)\rvert/Q(s)}{q(s)} \, ds
\end{equation}
is proved in Proposition \ref{rhopowerslipschitzthm}. 
Finally the monotonicity property follows directly from the above inequality, which concludes the proof.
 \end{proof}

\section{Breakdown for the EPDiff equation}\label{sec:breakdownEPDiff}
In this section we will apply our general breakdown theorem to obtain our breakdown results for the $H^1$ metric in any dimension and the $\dot{H}^2$ and $H^2$ metrics in dimension $n\ge 3$, having already obtained the full breakdown result in the $\dot{H}^1$ case in Section~\ref{sec:radialHS}.

\subsection{The EPDiff equation of the homogeneous $\dot H^2$-metric on $\R^n$}\label{sec:homH2_generaldim}
We aim to prove breakdown of the EPDiff equation on $\mathbb R^n$ for radial solutions with the inertia operator $A=\Delta^2$ and $n\geq 3$. We want to emphasize that the analysis of this equation has to be taken with caution, as the kernel of the inertia operator $A$ can lead to significant difficulties, see also the comments in Remark \ref{rem:hom_metrics}.
\begin{theorem}\label{thm:H2dotbreakdown}
Let $n\geq 3$ and suppose that the initial momentum   satisfies $\omega_0(r)\le 0$ for all $r\ge 0$.
Then the solution of equation \eqref{gammadeltaform} with homogeneous $\dot{H}^2$ operator $\Delta^2$ with $\gamma(0,r) = r$
and $\gamma_r(0,r)=\rho(0,r)=1$
breaks down in finite time, in the sense that $\gamma_r(t,r)=\rho(t,r)=0$ for some $t>0$ and $r\ge 0$.
\end{theorem}

\begin{proof}[Proof of Theorem~\ref{thm:H2dotbreakdown}]
We aim to apply Theorem~\ref{Qtheorem}: 
The required local existence of solutions is guaranteed by Corollary~\ref{localexistenceC0}.
The positivity condition~\ref{main_cond1} follows from
\begin{equation*}
\varphi(r,s)=\frac{s^{-n}( (n+2) s^{2} - (n-2) r^2)}{2n(n-2)(n+2)}> \frac{s^{-n}( s^{2} -r^2)}{2n(n+2)} \ge 0,
\end{equation*}
for $(r,s)\in D$. %
For the log-supermodularity condition~\ref{main_cond4} we compute
\begin{equation}
   \frac{\partial^2}{\partial r\partial s} \ln{\varphi(r,s)} = \frac{4rs(n^2-4)}{\big[ (n+2) s^2 - (n-2) r^2\big]^2},
\end{equation}
which is obviously nonnegative for $n\ge 3$.
 For the last condition~\ref{main_cond5} we obtain after a straightforward computation that $S(r) = \frac{2(n-2)}{n+2}$. %
Thus all conditions of Theorem~\ref{Qtheorem} are satisfied. It remains to calculate the function $Q(r)$, which is given by $Q(r) =2n(n-2)r^{n-3}$.
As $Q(r)\neq 0$ this implies that $\rho$ reaches zero in finite time and thus we obtain the desired breakdown result.
\end{proof}

\subsection{The higher dimensional Camassa-Holm equation.}\label{sec:CH}
The higher dimensional Camassa-Holm equation corresponds to the EPDiff equation of the right invariant $H^1$-metric, i.e.,  it corresponds to equation~\eqref{eq:EPDiff} with $A=1-\Delta$ on $\mathbb R^n$.
In the following we will refer to this equation restricted to radial solutions as the radial, $n$-dimensional Camassa-Holm equation.
Breakdown for smooth solutions of this family of equations was already shown in~\cite{li2013euler}, Theorem 2.2. We present the result here since it is a simple application of our general technique. 

\begin{theorem}\label{thm:H1breakdown}
Suppose the initial momentum $\omega_0$ satisfies $\omega_0(r)\le 0$ for all $r\ge 0$.
Then the solution of equation \eqref{gammadeltaform} with $H^1$ operator $(1-\Delta)$ with $\gamma(0,r) = r$
and $\gamma_r(0,r)=\rho(0,r)=1$
breaks down in finite time, in the sense that $\gamma_r(t,r)=\rho(t,r)=0$ for some $t>0$ and $r\ge 0$.
\end{theorem}
\begin{proof}
The required local existence follows from  Proposition~\ref{prop:wellposedness}.
Next we recall from Lemma \ref{lem:greensfunctionnonhomog} that the function $\varphi$ from formula \eqref{H1fullinversion} is given by $\varphi(r,s) = \funca_n(r)\funcb_n(s)$.
We aim to apply Theorem~\ref{Qtheorem}: we will use several properties of these functions that are discussed in detail in Appendix \ref{greenfunctionproof}.
The positivity condition~\ref{main_cond1} follows directly from the positivity of the Bessel functions. %
Since $\varphi$ is a product of two functions of one variable, the log-supermodularity condition~\ref{main_cond4} is trivially satisfied.

For the last condition~\ref{main_cond5} we need to compute $S(r)$, and for this we compute first
$$\partial_1\varphi(r,r) = \funca_n'(r) \funcb_n(r) \qquad \text{and}\qquad \varphi(0,r) = \funcb_n(r),$$
using the fact that $\funca_n(0)=1$ for all $n$ by \eqref{besselasympssmall}. 
Then $S(r)$ simplifies to
$$    S(r) = \frac{r \funcb_n(r) \big[ \funcb_n(r)\funca_n'(r) - \funca_n(r) \funcb_n'(r)\big]}{\funcb_n(r)^2} 
    = \frac{1}{r^n \funcb_n(r)},
$$ %
using the Wronskian \eqref{besselwronskian}.
Since this function is always positive even at $r=0$ by the asymptotic formula \eqref{besselasympssmall} and approaches infinity as $r\to \infty$ by \eqref{besselasympsbig}, we see that $S(r)\ge C$ for some positive constant $C$ and all $r\ge 0$.

Thus all conditions of Theorem~\ref{Qtheorem} are satisfied, %
and this implies that $\rho$ reaches zero in finite time. Thus the radial, $n$-dimensional Camassa-Holm equation blows up.
\end{proof}

\subsection{The EPDiff equation of the $H^2$-metric.}\label{sec:H2}
We are now ready to tackle our main breakdown result for the EPDiff equation: breakdown of the  EPDiff equation~\eqref{eq:EPDiff} corresponding to $A=(1-\Delta)^2$ in dimensions $n\geq 3$. 
\begin{theorem}\label{thm:H2breakdown}
Let $n\geq 3$ and suppose that the initial momentum  $\omega_0$ satisfies $\omega_0(r)\le 0$ for all $r\ge 0$.
Then the solution of equation the $n$-dimensional EPDiff equation~\eqref{gammadeltaform} with $H^2$ inertia operator $(1-\Delta)^2$ and initial conditions $\gamma(0,r) = r$
and $\gamma_r(0,r)=\rho(0,r)=1$
breaks down in finite time, in the sense that $\gamma_r(t,r)=\rho(t,r)=0$ for some $t>0$ and $r\ge 0$.
\end{theorem}

\begin{proof}[Proof of Theorem~\ref{thm:H2breakdown}]
The required local existence follows from  Proposition~\ref{prop:wellposedness}.
Thus it only remains to check that $$\varphi(r,s)=\tfrac{1}{2}\big[ n\funca_n(r) \funcb_n(s) + \tfrac{1}{n} \funca_n(r) \funcb_{n-2}(s) - n\funca_{n-2}(r) \funcb_n(s)\big],$$
 as defined in~\eqref{H2fullinversion}, satisfies all conditions of Theorem~\ref{Qtheorem}. 
This is slightly easier if we write
\begin{equation}\label{deltaH2factored}
\varphi(r,s) = \frac{n}{2} \funca_n(r) \funcb_n(s) \big[ 1 + \ratiob(s) - \ratioa(r)\big], \qquad \ratioa(r) = \frac{\funca_{n-2}(r)}{\funca_n(r)}, \qquad \ratiob(s) = \frac{\funcb_{n-2}(s)}{n^2\funcb_n(s)}.
\end{equation}
Again we will need several properties of Bessel functions discussed in Appendix \ref{greenfunctionproof}.

From formulas \eqref{besselprimeup}--\eqref{besselprimedown} we can compute that $\ratioa(r)$ satisfies the Riccati equation
$$
\ratioa'(r) = \tfrac{n}{r} \Big( \tfrac{r^2}{n^2} + \ratioa(r) - \ratioa(r)^2\big)  = \tfrac{n}{r} \Big( \tfrac{1}{2} + \sqrt{\tfrac{1}{4} + \tfrac{r^2}{n^2}} - \ratioa(r)\Big)
\Big(  -\tfrac{1}{2} + \sqrt{\tfrac{1}{4} + \tfrac{r^2}{n^2}} + \ratioa(r)\Big).
$$
Since $\ratioa(0) = 1$ by formula \eqref{besselasympssmall}, we can use this Riccati equation to conclude that
\begin{equation}\label{ratioainequalities}
\ratioa(r) \le  \sqrt{\tfrac{1}{4} + \tfrac{r^2}{n^2}} + \tfrac{1}{2} \qquad\text{and}\qquad \ratioa'(r)\ge 0 \quad\text{ for all $r\ge 0$.}
\end{equation}
These inequalities are strict as soon as $r>0$.
Similarly $\ratiob(r)$ satisfies the Riccati equation
\begin{equation}\label{ratiobriccati}
\ratiob'(r) = \tfrac{n}{r} \big( -\tfrac{r^2}{n^2} + \ratiob(r) + \ratiob(r)^2\big) = \tfrac{n}{r} \Big(
 \ratiob(r) +  \tfrac{1}{2} + \sqrt{\tfrac{1}{4} + \tfrac{r^2}{n^2}}\Big) \Big(
\ratiob(r) +  \tfrac{1}{2} - \sqrt{\tfrac{1}{4} + \tfrac{r^2}{n^2}}\Big),
\end{equation}
and since $\lim_{r\to 0} \ratiob(r)=0$ by \eqref{besselasympssmall}, we conclude that
\begin{equation}\label{ratiobinequalities}
\ratiob(r)\ge  \sqrt{\tfrac{1}{4} + \tfrac{r^2}{n^2}}-\tfrac{1}{2} \qquad\text{and}\qquad \ratiob'(r)\ge 0 \quad\text{ for all $r\ge 0$.}
\end{equation}
Similarly these are strict when $r>0$.
See Laforgia-Natalini~\cite{laforgia2010some} for the details.

To check condition~\ref{main_cond1}, i.e., that $\varphi(r,s)>0$ for all $(r,s)\in D$, we use the factored form \eqref{deltaH2factored} and the inequalities \eqref{ratioainequalities}--\eqref{ratiobinequalities} to get
$$ \varphi(r,s) \ge \frac{n}{2} \funca_n(r) \funcb_n(s) \Big[ \sqrt{\tfrac{1}{4} + \tfrac{s^2}{n^2}} - \sqrt{\tfrac{1}{4} + \tfrac{r^2}{n^2}} \Big] \ge 0.$$
The first inequality is strict as soon as $r>0$ and $s\ge r$, while the second is strict as soon as $s>r$ even for $r=0$.

To prove condition \ref{main_cond4}, that $\partial_r\partial_s \ln{\varphi(r,s)}\ge 0$, we observe that
\begin{align*}
\frac{\partial^2}{\partial r\partial s} \ln{\varphi(r,s)} &= \frac{\partial^2}{\partial r\partial s}\Big(
\ln{\big[ \funca_n(r)\big]} + \ln{\big[ \funcb_n(s)\big]} + \ln{\big[ 1 + \ratiob(s) - \ratioa(r)\big]}\Big) \\
&= \frac{\partial}{\partial r} \frac{\ratiob'(s)}{1 +  \ratiob(s) - \ratioa(r)} = \frac{\ratioa'(r) \ratiob'(s)}{\big[ 1+\ratiob(s)-\ratioa(r)\big]^2}
\end{align*}
which is nonnegative using \eqref{ratioainequalities}--\eqref{ratiobinequalities}.

It remains to show that $S(r)$, as defined in \ref{main_cond5}, is bounded below by
a positive constant for all $r\in [0,\infty)$, i.e., that for some $C>0$ we have
\begin{equation}\label{numerator}
\numerator(r) := r\big[\partial_1\varphi(r,r) \varphi(0,r) - \varphi(r,r) \partial_2\varphi(0,r)\big] \ge C \varphi(0,r)^2 \text{ for all $r\ge 0$.}
\end{equation}
Therefore we calculate the terms appearing here. For convenience we define the function
$$ \jfunc(r) = n^2\big[ \funca_{n-2}(r) - \funca_n(r)\big],$$
so that formula \eqref{H2fullinversion} becomes
$$ \varphi(r,s) = \tfrac{1}{2n} \big[ \funca_n(r)\funcb_{n-2}(s) - \jfunc(r) \funcb_n(s)\big].$$
Since $\funca_n(0)=1$ for all $n$ by \eqref{besselasympssmall}, we have $\jfunc(0)=0$, and we obtain 
$$ \varphi(0,r) = \tfrac{1}{2n} \, \funcb_{n-2}(r), \qquad \partial_2\varphi(0,r) = -\tfrac{r}{2} \, \funcb_n(r).$$

Recalling the derivative formulas \eqref{besselprimeup}--\eqref{besselprimedown}, we compute that
\begin{equation}\label{jderivative}
r \jfunc'(r) = -n \jfunc(r) + n r^2 \funca_n(r), \qquad r \alpha_n'(r) = \tfrac{1}{n} \jfunc(r),
\end{equation}
and thus the quantity $r\partial_1\varphi(r,r)$ simplifies to
\begin{align*}
r \partial_1\varphi(r,r) &= \tfrac{1}{2n} \big[ r\funca_n'(r)\funcb_{n-2}(r) - r\jfunc'(r) \funcb_n(r)\big] \\
&= \tfrac{1}{2n} \big[ \tfrac{1}{n} \jfunc(r) \funcb_{n-2}(r) + n\jfunc(r) \funcb_n(r) - n r^2 \funca_n(r) \funcb_n(r)\big].
\end{align*}

Plugging into the quantity $\numerator(r)$ from \eqref{numerator}, we get
\begin{align*}
\numerator(r) &= \tfrac{1}{4n^2} \Big( \funcb_{n-2}(r)\big[ \tfrac{1}{n}\jfunc(r) \funcb_{n-2}(r) + n\jfunc(r) \funcb_n(r) - nr^2 \funca_n(r) \funcb_n(r)\big] \\
&\qquad\qquad+ nr^2 \funcb_n(r) \big[ \funca_n(r) \funcb_{n-2}(r) - \jfunc(r) \funcb_n(r)\big]\Big) \\
&= \tfrac{1}{4n^2} \jfunc(r) \Big( \tfrac{1}{n} \funcb_{n-2}(r)^2 + n \funcb_{n-2}(r)\funcb_n(r) - nr^2 \funcb_n(r)^2 \Big).
\end{align*}
 Thus $S(r)$ is given by
$$ S(r) =  \frac{\jfunc(r) \big[ \tfrac{1}{n} \funcb_{n-2}(r)^2 + n \funcb_{n-2}(r)\funcb_n(r) - n r^2 \funcb_n(r)^2 \big]}{\funcb_{n-2}(r)^2}
= \frac{\jfunc(r)}{n \ratiob(r)^2} \big[ \ratiob(r)^2 + \ratiob(r) - \tfrac{r^2}{n^2}\big].$$
Since  $\jfunc(r) = nr\funca_n'(r) > 0$ for $r>0$ by \eqref{ratioainequalities}, and since 
the term in square brackets is positive for $r>0$ by \eqref{ratiobinequalities}, we conclude that $S(r)>0$ for all $r> 0$.

To show $S(r)$ is bounded below by a positive constant $C$, it is sufficient to show that $S(r)$ cannot approach zero as $r\to 0$ or as $r\to\infty$.
For this purpose we use the asymptotic formulas \eqref{besselasympssmall}--\eqref{besselasympsbig}.
By formula \eqref{besselasympssmall} we know that $\lim_{r\to 0} r^n \funcb_n(r) = \tfrac{1}{n}$, so that
\begin{align*}
\lim_{r\to 0} S(r) &= \lim_{r\to 0} \frac{j(r)}{r^2} \, \lim_{r\to 0} \left( \frac{r^2}{n}  + n \,\frac{r^n\funcb_n(r)}{r^{n-2}\funcb_{n-2}(r)} - n\frac{r^{2n} \funcb_n(r)^2}{r^{2n-4} \funcb_{n-2}(r)^2} \right) \\
&= \frac{n}{n+2} \left[ n \Big(\frac{n-2}{n}\Big) - n \Big( \frac{n-2}{n}\Big)^2 \right] \\
%
&= \frac{2(n-2)}{n+2}.
\end{align*}
Here we computed $\lim_{r\to 0} \frac{j(r)}{r^2}=\frac{n}{n+2}$ using L'Hopital's rule and formula \eqref{jderivative}. For $n\ge 3$ we see that $\lim_{r\to 0} S(r)$ is positive.

On the other hand as $r\to \infty$ it is easy to see from \eqref{besselasympsbig} that $S(r)\to \infty$ like $r^m e^r$ for some power $m$, and in particular it does not approach zero. So $S(r)$ has a positive minimum value at some $r\ge 0$, and condition \ref{main_cond5} is verified.
\end{proof}

\section{Future Work and Conclusions}\label{sec:conclusion}
\subsection*{A summary of global existence and breakdown results for the EPDiff equations:}
By the results of this article and the breakdown results of~\cite{chae2012blow,li2013euler}, we know that solutions to the $H^k$ and $\dot{H}^k$ EPDiff equations blow up in any dimension if $k\in\{0,1\}$, and that they blow up
in dimension $n\ge 3$ for $k=2$. Combining this  with the global existence results for $k>\frac{n}{2}+1$~\cite{escher2014geodesic,mumford2013euler,bauer2015local,bauer2023regularity,ebin1970groups}, this gives a complete characterization for which (integer order) EPDiff equations all solutions exist for all time if $n=1$ or $n=3$. 
Using the geometric interpretation of the EPDiff equation as an Euler-Arnold equation on the diffeomorphism group, cf. Section~\ref{sec:EulerArnold}, directly leads to the following characterization of geodesic (in)completeness for the corresponding diffeomorphism group in dimensions one and three:
\begin{corollary}\label{cor:geodesiccomplete}
Consider the diffeomorphism group $\operatorname{Diff}(\mathbb R^n)$  equipped with the right-invariant Sobolev metric $G^k$ of order $k\in \mathbb N$. We have:
\begin{itemize}
    \item if $n=3$ then $(\operatorname{Diff}(\mathbb R^3),G^k)$ is geodesically complete if and only if $k\geq 3$, i.e., for any $k\geq 3$ and any initial conditions $U_0\in H^{\infty}(\mathbb R^3,\mathbb R^3)$, the solution to the geodesic equation (EPDiff equation, resp.) exists for all time $t$, whereas for any $k\in \{0,1,2\}$, there exist initial conditions $U_0\in H^{\infty}(\mathbb R^3,\mathbb R^3)$ such that the solution breaks down in finite time.
    \item if $n=1$, then $(\operatorname{Diff}(\mathbb R),G^k)$ is geodesically complete if and only if $k\geq 2$. 
\end{itemize}
\end{corollary}

\subsection*{Open questions for EPDiff equations I: critical indices and higher dimensions}
Note that we did not obtain a full characterization for which integers $k$ the corresponding EPDiff equations are globally well-posed if $n=2$ or if $n>3$. In dimension two it only remains to resolve the case $k=2$: the difficulty here is that this is exactly the critical index for the Sobolev embedding theorem---$H^2$ functions are almost but not quite guaranteed to be $C^1$---and the present technique does not work in that case. We have explicitly computed the quantity $S(r)$ in any dimension and found that its value at $r=0$ is proportional to $(n-2)$; in particular it is strictly positive when $n>2$ but approaches zero at one point in dimension $n=2$. One might at first think this is a mild obstacle, but in fact it turns out to change everything: the evidence based on preliminary estimates is that solutions exist globally.  In dimension $n>3$, it remains to prove breakdown for higher integer values of $k<\frac{n}{2}+1$. Based on the pattern here, we conjecture the following:
\begin{conjecture}\label{integerconjecture}
    For $k\in \mathbb{N}$, all solutions of the Euler-Arnold equation \eqref{mainomega} exist for all time if $k\geq \frac{n}{2}+1$, while there are always some solutions that break down in finite time if $k<\frac{n}{2}+1$. Consequently, the space $(\operatorname{Diff}(\mathbb R^n),G^k)$ is geodesically complete if and only if $k\geq \frac{n}{2}+1$. 
\end{conjecture}
We believe the extension of the breakdown results to higher $k$ ($n$, resp.) to be computationally involved but rather straightforward. We foresee bigger difficulties in obtaining the global existence for the critical Sobolev indices $k=\frac{n}{2}+1$ with $n$ even: 
it will likely require substantially more complicated computations to extend the analysis in the one-dimensional case as in \cite{bauer2020geodesic} to general (non-radial) initial conditions.

\subsection*{Open questions for EPDiff equations II: fractional orders}
Beyond this we can study noninteger values of $k$, where the operator $(\aconst - \Delta)^k$ can be defined using the Fourier transform. These situations are indeed very relevant in modeling fluid mechanics; for example the case $k=1/2$ in dimension $n=1$ is related to the De Gregorio equation~\cite{de1990one,de1996partial} and to the Okamoto-Sakajo-Wunsch family of equations~\cite{okamoto2008generalization}; and in dimension two the case $k=-\frac{1}{2}$ is related to the surface quasigeostrophic (SQG) equation. The main difference between these situations and the integer order metrics studied in the present article is that the kernel $\delta(r,s)$ factors as a sum of products of functions of $r$ and $s$ separately in the integer case, while no such factorization is possible in general. Nonetheless we believe that our general technique will probably apply in the fractional case, and that we can extend Conjecture \ref{integerconjecture} word for word to the case of any real $k$.  Note that when $n=1$, the critical index is $k=\frac32$, which is related to the Weil-Petersson metric~\cite{gay2015geometry}. In this case, it has been shown that the equations are indeed globally wellposed; see \cite{PrestonWashabaugh, bauer2020geodesic}. Note that for $k< \tfrac12$ we do not obtain an ODE on a Banach space, e.g., for $k=0$ this has been shown by Constantin and Kolev in~\cite{constantin_kolev}. This will require a different breakdown analysis for these low-order cases, as the ODE interpretation is a central ingredient of the current criterion.

\appendix

\section{Proofs of Lemmas \ref{lem:greensfunctionhomog} and \ref{lem:greensfunctionnonhomog}, Green's functions for the Laplacian}\label{greenfunctionproof}

\begin{proof}[Proof of Lemma \ref{lem:greensfunctionhomog}]
{We want to prove that if $j$ is a nonnegative integer and $\power$ is a real number such that $0<\power<n$, then $\omega \in Q^{j}_{\power+1}$ implies that there is a unique solution 
$u \in Q^{j+2}_{\power-1}$ of the equation $\Delta(u\partial_r) = -\omega\partial_r$. Assume $\omega\in \mathcal{Q}^{j}_{\power+1}$; recall from the definition \eqref{Qspacedef} that $\displaystyle \limsup_{r\to\infty} r^{\power+1}\lvert \omega(r)\rvert < \infty$, and that $\omega^{(j)}$ is locally in $L^1$, with $\omega^{(2i)}(0)=0$ for $0\le 2i<j$.
We want to prove that the solution $u$ will satisfy 
$\displaystyle \limsup_{r\to\infty} r^{\power-1} \lvert u(r)\rvert < \infty,$
and that $u^{(j+2)}$ is locally $L^1$, with 
$u^{(2i)}(0)=0$ for $0\le 2i\le j$.}

{Recalling formula \eqref{eq:vectorLaplace}, it is easy to see that we may rewrite the equation $\Delta(u\partial_r)=\omega\partial_r$ as 
\begin{equation}\label{ubarformhomog} 
\frac{1}{r^{n+1}} \, \frac{d}{dr} \Big(r^{n+1} \ubar'(r)\Big) = -\frac{\omega(r)}{r}, \qquad \ubar(r):=\frac{u(r)}{r}.
\end{equation}
We are demanding that $u(0)=0$ and $u$ is at least $C^1$, which implies that 
$$ \lim_{\sigma\to 0^+} \sigma^{n+1}\ubar'(\sigma) = \lim_{\sigma\to 0^+} \sigma^n \left( u'(\sigma) - \frac{u(\sigma)}{\sigma}\right) = 0,$$
in addition to implying that $\ubar(0)$ is finite. As such, multiplying \eqref{ubarformhomog} by $r^{n+1}$ and integrating from $r=0$, 
%
%
we obtain
\begin{equation}\label{ubarprimeformula}
    \sigma^{n+1} \overline{u}'(\sigma) = -\int_0^{\sigma} s^n \omega(s)\,ds.
    \end{equation}
}

{Now integrate from $0$ to $r$, then change integration order by Fubini's Theorem (since $\omega$ is at least locally $L^1$) to get 
$$
\ubar(r) - \ubar(0) = - \int_0^r \int_0^{\sigma} \frac{s^n\omega(s)}{\sigma^{n+1}} \, ds\,d\sigma 
= - \int_0^r \int_s^r \frac{s^n \omega(s)}{\sigma^{n+1}} \, d\sigma \, ds
= \frac{1}{n} \int_0^r s^n \omega(s) (r^{-n} - s^{-n}) \, ds, $$
which leads to
\begin{equation}\label{ubarformula}
    \ubar(r) = \ubar(0)  - \frac{1}{n} \int_0^r \omega(s)\,ds + \frac{1}{nr^n} \int_0^r s^n \omega(s) \, ds.
\end{equation}}

{At the moment $\ubar(0)$ is undetermined. To find it we observe that by our assumptions on $\omega$, 
\begin{equation}\label{decayspaceexplicit}
\exists R>0, \exists C>0 \text{ s.t. } r\ge R \; \Longrightarrow \lvert \omega(r)\rvert \le \frac{C}{r^{\power+1}}.
\end{equation}
Since $\power>0$, this implies that $\omega$ is $L^1$ on all of $[0,\infty)$, not just locally, and that the limit of the middle term in \eqref{ubarformula} is finite. On the other hand we can show that the limit of the last term is zero: in fact that term is in $\decayspace_{\power}(\mathbb R_{\geq 0},\mathbb R)$ since
for $r\ge R$ we have 
\begin{align*}
    r^{\power} \left\lvert r^{-n}\int_0^r s^n \omega(s)\,ds\right\rvert 
    &\le r^{\power-n} \left(  \int_0^R s^n \lvert \omega(s)\rvert \,ds + C \int_R^r s^{n-1-\power}\,ds \right) \\
    &\le r^{\power-n} \left( \int_0^R s^n \lvert \omega(s)\rvert \,ds + \frac{C}{n-\power} (r^{n-\power}-R^{n-\power}) \right)  \\
    &\le M r^{\power-n} + \frac{C}{n-\power}, \qquad M:=\int_0^R s^n\lvert \omega(s)\rvert \, ds - \frac{C}{n-\power} R^{n-\power}
\end{align*}
which is bounded for $r\ge R$ since $n>\power$.}

{Thus the only way to hope for $\ubar(r)$ to decay (and thus for $u(r)$ to be in $\decayspace_{\power+1}$) is if 
\begin{equation}\label{ubarzero}
    \overline{u}(0) = \frac{1}{n} \int_0^{\infty} \omega(s)\,ds,
\end{equation}
which then turns \eqref{ubarformula} into 
\begin{equation}\label{newubarformula}
\ubar(r) = \frac{1}{n r^n} \int_0^r s^n \omega(s)\,ds + \frac{1}{n} \int_r^{\infty} \omega(s)\,ds; 
\end{equation}
thus if there is a solution, it must satisfy this equation. It remains to check that $u(r)=r\ubar(r)$ satisfies the desired conditions: we need $u\in \decayspace_{\power-1}$ (which is equivalent to $\ubar\in \decayspace_{\power}$), and we need 
$u\in W^{j+2,1}_{\text{odd}}$. For decay, we have already seen that the first term in \eqref{newubarformula} is in $\decayspace_{\power}$, so we just need to check the second term. For $r\ge R$, we have by \eqref{decayspaceexplicit} that 
\begin{align*}
    r^{\power} \left\lvert \int_r^{\infty} \omega(s) \,ds \right\rvert 
    &\le C r^{\power} \int_r^{\infty} \frac{ds}{s^{\power+1}} = \frac{C}{\power}, 
\end{align*}
which is precisely what we need. Hence the formula \eqref{newubarformula} defines a function $\ubar$ which is in the correct decay space $\decayspace_{\power}$, and thus $u\in\decayspace_{\power-1}$.
}

{Finally we need to check that $\omega\in W^{j,1}_{\text{odd}}$ implies that $u\in W^{j+2,1}_{\text{odd}}$. This is easiest to do using the formula \eqref{ubarprimeformula}. In fact it is easy to show by induction that for any nonnegative integer $i$, if $\omega^{(i)}\in L^1[0,r]$, then we have 
\begin{equation}\label{ubarinduction}
\ubar^{(i+1)}(r) = -r^{-n-i-1} \int_0^r s^{n+i} \omega^{(i)}(s)\,ds.
\end{equation}
The base case is precisely \eqref{ubarprimeformula}, and the inductive step follows by an integration by parts and the product rule. We immediately conclude that if $\omega^{(j)}$ is locally in $L^1$, then $\ubar^{(j+1)}$  and thus $u^{(j+1)}$ are absolutely continuous for $r>0$. 
}

{All that remains is to check that $u^{(2i)}(0)=0$ if $2i<j$. For $i=0$ we get that $u(0) = r \ubar(0)=0$ by \eqref{ubarzero}. The iterated product rule on $u(r)=r\ubar(r)$ shows that for $i\ge 1$, we have $u^{(2i)}(0)=0$ if and only if $\ubar^{(2i-1)}(0)=0$, and 
an integration by parts of \eqref{ubarinduction} with $i$ replaced by $2i-2$ shows that 
$$ \ubar^{(2i-1)}(0) = \frac{1}{n+2i-1} \left( -\omega^{(2i-2)}(r) + r^{-n-2i+1} \int_0^r s^{n+2i-1} \omega^{(2i-1)}(s)\,ds\right).$$
Since $2i-1\le j$, we see that $\omega^{(2i-1)}(s)$ is locally in $L^1$, so the second term in the numerator approaches zero as $r\to 0$, while the first term is zero by our assumption $\omega\in W^{j,1}_{\text{odd}}$. 
}

{Formula \eqref{newubarformula} yields precisely 
\eqref{generalwsoln} in the case \eqref{H1dotinversion} which covers the case $k=1$ for any $n\ge 1$, since we can find $\power$ such that $0<\power<n$. We have seen that if $\omega\in \decayspace_{\power+1}$ and $0<\power<n$, then $\Delta^{-1}\omega\in \decayspace_{\power-1}$. As long as $0<\power-2<n$ as well, we can continue this, and obtain an iterated formula for $\Delta^{-2}$; note that this now requires $n\ge 3$, and similarly operators $\Delta^{-k}$ are only defined on spaces with decay condition $\decayspace_{\power+1}$ where $2(k-1)<\power<n$.
}

{We can easily obtain the explicit formula for the solution of $\Delta^2(u(r)\,\partial_r) = \omega(r)\,\partial_r$ by iterating the formula above:
using $K_1(r,s) := \delta_1\big(\min\{r,s\}, \max\{r,s\}\big)$ to denote the Green function for $\Delta^{-1}$, we can write 
$$ \Delta^{-1}(\omega)(r) = \int_0^{\infty} K_1(r,s) \, s^{n-1}\omega(s)\,ds.$$
Iterating this we get 
$$ 
\Delta^{-2}(\omega)(r) = \int_0^{\infty}   \left[\int_0^{\infty} \sigma^{n-1} K_1(r,\sigma) K_1(\sigma,s) \, d\sigma \right] s^{n-1}\omega(s)\,ds
= \int_0^{\infty} K_2(r,s) \, s^{n-1}\omega(s)\,ds,
$$ 
where 
$$ K_2(r,s) = \delta_2(r,s) = rs \varphi_2(r,s) \text{ for $r\le s$},$$
and 
$$ \varphi_2(r,s) = \int_0^{\infty} \sigma^{n+1} \varphi_1\big(\min\{r,\sigma\}, \max\{r,\sigma\}\big) \varphi_1\big(\min\{\sigma,s\}, \max\{\sigma,s\}\big) \, d\sigma.$$
This integral is easy to compute for $r\le s$ using $\varphi_1(r,s) = \tfrac{1}{n} s^{-n}$ by breaking up the intervals, and we get 
\begin{align*} 
\varphi_2(r,s) &= \tfrac{1}{n^2} \int_0^r \sigma^{n+1} r^{-n} s^{-n} \,d\sigma + \tfrac{1}{n^2} \int_r^s \sigma^{n+1} \sigma^{-n} s^{-n} \,d\sigma + \tfrac{1}{n^2} \int_s^{\infty} \sigma^{n+1} \sigma^{-n}\sigma^{-n} \, d\sigma \\
&= \frac{s^{2-n}}{2n(n-2)} - \frac{r^2 s^{-n}}{2n(n+2)},
\end{align*}
which is \eqref{H2dotinversion}.} 
\end{proof}

\begin{proof}[Proof of Lemma \ref{lem:greensfunctionnonhomog}]
{For the nonhomogeneous operator $(1-\Delta)$, it is again simpler to solve the equation $(1-\OvDel)\ubar(r) = \overline{\omega}(r)$, as in the previous proof, where $u(r)=r\ubar(r)$ and $\omega(r)=r\overline{\omega}(r)$, with 
$\OvDel = \partial_r^2 + \frac{n+1}{r} \partial_r$.}
We need some special properties of the functions $\funca_p$ and $\funcb_p$ defined in \eqref{besseldef}.
Standard formulas (see e.g., Gradshteyn-Ryzhik~\cite{gradry} (8.486)) show that the derivatives satisfy
\begin{alignat}{3}
\funca_p'(r) &= \tfrac{r}{p+2} \funca_{p+2}(r), \qquad &\qquad \funcb_p'(r) &= -(p+2) r \funcb_{p+2}(r) \label{besselprimeup} \\
r\funca_p'(r) &= -p \funca_p(r) + p \funca_{p-2}(r), & r \funcb_p'(r) &= -p \funcb_p(r) - \tfrac{1}{p} \funcb_{p-2}(r). \label{besselprimedown}
\end{alignat}
From formulas \eqref{besselprimeup}--\eqref{besselprimedown} we can compute that
\begin{equation}\label{inductivelaplacebessel}
(1-\OvDel) \funca_p(r) = \frac{p-n}{p+2} \funca_{p+2}(r) \qquad \text{and}\qquad (1-\OvDel) \funcb_p(r) = (p+2) (n-p) \funcb_{p+2}(r),
\end{equation}
for any $p\in\mathbb{R}$ and any $n\in\mathbb{N}$, which makes these functions especially convenient
for handling the operator $(1-\OvDel)$. In particular we see that
\begin{equation}\label{fundysolns}
(1-\OvDel) \funca_n = 0, \qquad (1-\OvDel) \funca_{n-2} = -\tfrac{2}{n} \funca_n, \qquad (1-\OvDel)\funcb_n=0, \qquad
(1-\OvDel)\funcb_{n-2} = 2n \funcb_n,
\end{equation}
which shows that $\{\funca_n, \funcb_n, \funca_{n-2}, \funcb_{n-2}\}$ form a basis of fundamental solutions to the
$H^2$ inertia operator $(1-\OvDel)^2$.

The modified Bessel function identity
$$ I_{\nu}(r) K_{\nu+1}(r) + I_{\nu+1}(r) K_{\nu}(r) = \frac{1}{r}$$
(e.g., \cite{gradry} (8.477))
shows that the Wronskian of \eqref{besseldef} satisfies
\begin{equation}\label{besselwronskian}
\funcb_p(r) \funca_p'(r) - \funca_p(r) \funcb_p'(r) = r \Big[ \frac{\funcb_p(r) \funca_{p+2}(r)}{p+2} 
+ (p+2) \funca_p(r) \funcb_{p+2}(r) \Big] = r^{-p-1}.
\end{equation}
Standard formulas for Bessel functions show that for $p>0$ we have the asymptotic behavior
\begin{alignat}{3}
\funca_p(0) &= 1, \qquad &\qquad
\lim_{r\to 0} r^p \funcb_p(r) &= \tfrac{1}{p}, \label{besselasympssmall} \\
\lim_{r\to\infty} r^{(p+1)/2} e^{-r} \funca_p(r) &= \frac{\constyp}{\sqrt{2\pi}}, \qquad &
\lim_{r\to\infty} r^{(p+1)/2} e^r \funcb_p(r) &= \frac{\sqrt{2\pi}}{2\constyp}.\label{besselasympsbig}
\end{alignat}
In particular $\funca_p(r)$ is finite at $r=0$ and approaches infinity as $r\to\infty$, while $\funcb_p(r)$ does the opposite.
The nice behavior at $r=0$ is the reason for our choice of scaling in the definition \eqref{besseldef}. 

To solve the equation $(1-\OvDel)\ubar(r) = \omegabar(r)$ for $\ubar$, we may use the usual variation of parameters technique to write
\begin{equation}\label{wvarparH1def}
\ubar(r) = \funca_n(r) v_1(r) + \funcb_n(r) v_{2}(r),
\end{equation}
where the functions $v_1$ and $v_2$ satisfy the system
\begin{equation}\label{auxiliaryvarparH1}
\funca_n(r) v_1'(r) + \funcb_n(r) v_2'(r) = 0, \qquad\qquad  \funca_n'(r) v_1'(r) + \funcb_n'(r) v_2'(r) = -\omega(r)/r.
\end{equation}
Using the Wronskian from \eqref{besselwronskian}, the solution of \eqref{auxiliaryvarparH1} is
\begin{equation}\label{uvvarparH1}
v_1'(r) = -r^n \funcb_n(r) \omega(r), \qquad v_2'(r) = r^n \funca_n(r) \omega(r).
\end{equation}
As in the $\dot{H}^1$ case, the asymptotic behavior of $\funca_n$ and $\funcb_n$ implies for finiteness of $u(r)$ that we must have $v_1(r)\to 0$ as $r\to\infty$
and $v_2(0)=0$, so integrating \eqref{uvvarparH1} yields
$$ v_1(r) = \int_r^{\infty}  \funcb_n(s) s^n \omega(s) \,ds, \qquad v_2(r) = \int_0^r \funca_n(s) s^n \omega(s)\,ds.$$
Then plugging into \eqref{wvarparH1def} yields 
$$ u(r) = r\ubar(r) = \funcb_n(r) \int_0^r \funca_n(s) r s^n \omega(s)\,ds + \funca_n(r) \int_r^{\infty} \funcb_n(s) r s^n \omega(s)\,ds,$$
which is of the form \eqref{generalwsoln} with $\delta(r,s) = rs \funca_n(r) \funcb_n(s)$, proving \eqref{H1fullinversion}.

Finally we consider the solution of $(1-\OvDel)^2 \ubar(r) = \omegabar(r)$. We could iterate the formula for $(1-\OvDel)^{-1}$, but this results in rather difficult integrals of Bessel functions with a number of terms that cancel out anyway. A simpler method is variation of parameters using the fundamental solutions \eqref{fundysolns}. This gives
\begin{equation}\label{varparH2}
\ubar(r) =  \funca_n(r) v_1(r) + \funcb_n(r) v_2(r) + \funca_{n-2}(r) v_3(r) + \funcb_{n-2}(r) v_4(r),
\end{equation}
with boundary conditions $v_1(r), v_3(r)\to 0$ as $r\to\infty$ and $v_2(0)=v_4(0)=0$, as in the $H^1$ case above. The standard  technique gives
the auxiliary conditions
\begin{equation}\label{longauxiliaryH2}
\begin{split}
\funca_n(r) v_1'(r) + \funcb_n(r) v_2'(r) + \funca_{n-2}(r) v_3'(r) + \funcb_{n-2}(r) v_4'(r) &= 0, \\
\funca_n'(r) v_1'(r) + \funcb_n'(r) v_2'(r) + \funca_{n-2}'(r) v_3'(r) + \funcb_{n-2}'(r) v_4'(r) &= 0.
\end{split}
\end{equation}
Since $\ubar$, $\ubar'$, and $\ubar''$ only involve differentiating the functions $\funca$ and $\funcb$, we obtain via \eqref{fundysolns} that
\begin{align*}
(1-\OvDel)\ubar(r) &= \big[ (1-\OvDel)\funca_n(r)\big] v_1(r)
+ \big[ (1-\OvDel)\funcb_n(r)\big] v_2(r)  \\
&\qquad\qquad+ \big[ (1-\OvDel)\funca_{n-2}(r)\big] v_3(r) + \big[ (1-\OvDel)\funcb_{n-2}(r)\big] v_4(r) \\
&= -\tfrac{2}{n} \funca_n(r) v_3(r) + 2n \funcb_n(r) v_4(r).
\end{align*}

We therefore impose one more auxiliary condition
\begin{equation}\label{3rdauxiliaryH2}
-\tfrac{2}{n} \funca_n(r) v_3'(r) + 2n \funcb_n(r) v_4'(r) = 0,
\end{equation}
to obtain
$$  \frac{d}{dr} (1-\OvDel) \ubar(r) =  -\tfrac{2}{n} \funca_n'(r) v_3(r) + 2n \funcb_n'(r) v_4(r),$$
and finally get the equation
\begin{equation}\label{finalH2green}
(1-\OvDel)^2 \ubar(r) =  \tfrac{2}{n}\funca_n'(r) v_3'(r) - 2n \funcb_n'(r) v_4'(r) = \omega(r)/r.
\end{equation}

Using the Wronskian \eqref{besselwronskian}, the solution of the system \eqref{3rdauxiliaryH2}--\eqref{finalH2green} is given by
\begin{equation}\label{u2v2solnH2green}
v_3'(r) = \tfrac{n}{2} r^n \funcb_n(r) \omega(r), \qquad \qquad v_4'(r) = \tfrac{1}{2n} r^n \funca_n(r) \omega(r),
\end{equation}
similarly to the $H^1$ case \eqref{uvvarparH1}.
We now use these in \eqref{longauxiliaryH2} to solve for $v_1'$ and $v_2'$, which become (using
\eqref{besselprimeup} and \eqref{besselwronskian} to simplify) the system
\begin{align*}
\funca_n(r) v_1'(r) + \funcb_n(r) v_2'(r) &= -\tfrac{1}{2} \omega(r), \\
\funca_n'(r) v_1'(r) + \funcb_n'(r) v_2'(r)  &= 0,
\end{align*}
with solution
$$ v_1'(r) = \tfrac{1}{2} r^{n+1} \funcb_n'(r) \omega(r), \qquad \qquad v_2'(r) = -\tfrac{1}{2} r^{n+1} \funca_n'(r) \omega(r).$$
Equivalently using \eqref{besselprimedown} we may rewrite this as
\begin{equation}\label{u1v1solnH2green}
v_1'(r) = -\tfrac{1}{2} r^n \big[ n \funcb_n(r) + \tfrac{1}{n}\funcb_{n-2}(r)\big] \omega(r), \qquad\qquad 
v_2'(r) = \tfrac{n}{2} r^n \big[ \funca_n(r) - \funca_{n-2}(r)\big] \omega(r).
\end{equation}
Now using the boundary conditions $v_i(r)\to 0$ as $r\to\infty$ for odd $i$ and $v_i(0)=0$ for even $i$, we can integrate \eqref{u2v2solnH2green}
and \eqref{u1v1solnH2green}, plug in to formula \eqref{varparH2}, and obtain the solution
\begin{multline*} \ubar(r) = 
\tfrac{1}{2} \funca_n(r) \int_r^{\infty}\big[ n \funcb_n(s) + \tfrac{1}{n}\funcb_{n-2}(s)\big] s^n \omega(s) \,ds
 + \tfrac{n}{2}  \funcb_n(r) \int_0^r \big[ \funca_n(s) - \funca_{n-2}(s)\big] s^n \omega(s) \,ds \\
 - \tfrac{n}{2} \funca_{n-2}(r) \int_r^{\infty} \funcb_n(s) s^n \omega(s) \,ds 
 + \tfrac{1}{2n} \funcb_{n-2}(r) \int_0^r \funca_n(s) s^n \omega(s) \,ds 
 \end{multline*}
Multiplying by $r$ to get $u(r)$ puts this in the form \eqref{generalwsoln} with $\delta(r,s)$ given by \eqref{H2fullinversion}.
\end{proof}

\section{Local existence in the space of radial $C^1$ diffeomorphisms}\label{sec:app_well}

In this section we consider the local existence for the equation \eqref{rhoODE} on the space of  continuous functions on $[0,\infty)$ 
bounded above and below by positive constants, as
discussed at the end of Section \ref{sec:vort_trans}, denoted by
\begin{equation}\label{spaceydef}
\spacey = \big\{ \rho\in C([0,\infty),\mathbb{R}_+) \, \big\vert\, \exists b\ge a>0 \text{ s.t. } \forall r\ge 0, \quad a\le \rho(r)\le b\big\}.
\end{equation}
Open balls in this space in the supremum norm are given by 
\begin{equation}\label{densityspace}
\Pspace_{a,b} := \Big\{ \rho\in C([0,\infty),\mathbb{R}_+) \, \Big\vert\, \exists a>0, b\ge a \text{ s.t. } a< \rho(r)< b \,\forall r\in [0,\infty)\Big\}.
\end{equation} 
In this Appendix we will prove that 
when $\delta(r,s) = r^p s^{-q}$ for $p,q>0$, the vector field \eqref{vectorfield} is a Lipschitz vector field on each $\Pspace_{a,b}$, and thus we have local existence on $\spacey$. Since $\rho(t,r) = \gamma_r(t,r)$, this shows existence of $C^1$ radial solutions in the $\dot{H}^1$ case for $n\ge 1$
and the $\dot{H}^2$ case for $n\ge 3$. 

Explicitly when $\delta(r,s) = r^p s^{-q}$, the vector field \eqref{vectorfield} becomes $X(\rho)= \rho \vectyF(\rho)$, where 
\begin{equation}\label{rhovecpowers}
\vectyF(\rho)(r) = -q \gamma(r)^{-q-1} \int_0^r \frac{\gamma(s)^p }{\rho(s)}  \, s^{n-1} \omega_0(s)\,ds + p\gamma(r)^{p-1} 
\int_r^{\infty} \frac{\gamma(s)^{-q}}{\rho(s)}  \, s^{n-1}\omega_0(s)\,ds. 
\end{equation}
Here we think of $\gamma$ as also being a function of $\rho$ defined by $\gamma(r) = \int_0^r \rho(s)\,ds$, while $\omega_0$ is independent of $\rho$, as in Section \ref{sec:vort_trans}.

To simplify matters we will write this in the form 
\begin{equation}\label{vecty}
\vectyF(\rho) = -q \vectyFone_{p,q+1}(\rho) + p \vectyFtwo_{p-1,q},
\end{equation}
where 
\begin{equation}\label{vecfields12}
\vectyFone_{p,q}(\rho)(r) = \gamma(r)^{-q} \int_0^r \frac{\gamma(s)^p }{\rho(s)}  \, z_0(s)\,ds \;\text{and}\;
\vectyFtwo_{p,q}(\rho)(r) = \gamma(r)^p  \int_r^{\infty} \frac{\gamma(s)^{-q}}{\rho(s)}  \, z_0(s)\,ds,
\end{equation}
with $z_0(s) = s^{n-1} \omega_0(s)$. 

\begin{proposition}\label{rhopowerslipschitzthm}
Suppose $p$ and $q$ are nonnegative integers and that $\int_0^{\infty} s^{p-q+n-1} \lvert \omega_0(s)\rvert \,ds<\infty$. Then for any positive real numbers $a<b$,  
the functions $\vectyFone_{p,q}$ and $\vectyFtwo_{p,q}$ defined by formula \eqref{vecfields12} map $\Pspace_{a,b}$ to $C([0,\infty))$ and are Lipschitz functions in the supremum topology. 
\end{proposition}

\begin{proof}
We start with $\vectyFone_{p,q}$. Let $\rho$ and $\zeta$ be two elements of $\Pspace_{a,b}$, so that $a< \rho(r)< b$ and $a< \zeta(r)< b$ for all $r\ge 0$. We will show that $\vectyFone$ is bounded and Lipschitz on this set. 
Let $\gamma$ and $\eta$ be the corresponding diffeomorphisms given by
$$ \gamma(r) := \int_0^r \rho(\sigma)\,d\sigma \qquad \text{and}\qquad \eta(r) := \int_0^r \zeta(\sigma)\,d\sigma.$$
We clearly have $ar \le \gamma(r)\le br$ for all $r\ge 0$. If $\lVert \rho - \zeta\rVert = \varepsilon$, then 
by definitions of $\gamma$ and $\eta$ we clearly have $\lvert \gamma(r)-\eta(r)\rvert \le \varepsilon r$ for all $r\ge 0$. 

First we establish boundedness for a particular $\rho$: we have 
$$ \lvert \vectyFone(\rho)(r)\rvert \le  \gamma(r)^{-q} \int_0^r \frac{\gamma(s)^p }{\rho(s)}  \, \lvert  z_0(s)\rvert \,ds
\le \frac{b^p r^{-q}}{a^{q+1}}\int_0^r s^p s^{n-1} \lvert \omega_0(s)\rvert \, ds \le \frac{b^p}{a^{q+1}} \int_0^r s^{p-q+n-1} \lvert \omega_0(s)\rvert \, ds,$$
since $r^{-q}\le s^{-q}$ for $s\in [0,r]$, using the fact that $q\ge 0$. 

Next we establish the Lipschitz bound. The difference $\vectyFone(\rho)-\vectyFone(\zeta)$ can be expressed as a sum of three terms:
$$ \vectyFone(\rho) - \vectyFone(\zeta) = I + II + III, $$
where 
\begin{align}
I(r) &= \big[ \gamma(r)^{-q} - \eta(r)^{-q} \big]\int_0^r \frac{z_0(s)\gamma(s)^p}{\rho(s)}\,ds  \label{Itermlip} \\ 
II(r) &= \eta(r)^{-q} \int_0^r \big[ \gamma(s)^p - \eta(s)^p \big] \, \frac{z_0(s)}{\rho(s)} \,ds \label{IItermlip} \\
III(r) &= \eta(r)^{-q} \int_0^r \eta(s)^p z_0(s) \Big[ \frac{1}{\rho(s)} - \frac{1}{\zeta(s)} \Big] \, ds.\label{IIItermlip}
\end{align}
We need to bound the three terms \eqref{Itermlip}--\eqref{IIItermlip} in terms of $\varepsilon$. 

For term $I$ in \eqref{Itermlip} we observe, using $ar\le \gamma(r),\eta(r)\le br$, that
\begin{align*}
\big\lvert \gamma(r)^{-q}-\eta(r)^{-q}\big\rvert &= \eta(r)^{-q}\gamma(r)^{-q} \big\lvert  \eta(r)^q - \gamma(r)^q\big\rvert \\
&= \eta(r)^{-q}\gamma(r)^{-q} \big\lvert \eta(r) - \gamma(r)\big\rvert \sum_{i=0}^{q-1} \eta(r)^i \gamma(r)^{q-1-i} \\
&\le (ar)^{-2q} \, (\varepsilon r) \, \sum_{i=0}^{q-1} (br)^i (br)^{q-1-i} \le \varepsilon q b^{q-1} a^{-2q} r^{-q}.
\end{align*}
As such we immediately obtain 
\begin{align*} 
\big\lvert I(r)\big\rvert &\le \frac{\varepsilon q b^{q-1}}{a^{2q}} r^{-q} \int_0^r \frac{\lvert z_0(s)\rvert\gamma(s)^p}{\rho(s)}\,ds \\
&\le  \frac{\varepsilon q b^{q-1}}{a^{2q}} r^{-q} \int_0^r \frac{ \lvert z_0(s)\rvert b^p s^p}{a}\,ds \le \frac{\varepsilon q b^{q+p-1}}{a^{2q+1}} \, r^{-q} \int_0^r s^{p} \lvert z_0(s)\rvert \,ds.
\end{align*}

For term $II$ in \eqref{IItermlip} a similar technique yields 
$$ \big\lvert \gamma(s)^p - \eta(s)^p \big\rvert 
= \big\lvert \gamma(s)-\eta(s)\big\rvert \sum_{i=0}^{p-1} \gamma(s)^i \eta(s)^{p-1-i} 
\le (\varepsilon s) p (bs)^{p-1} = \varepsilon p b^{p-1} s^p,$$
and thus 
$$ II(r) \le (ar)^{-q} \int_0^r \frac{\varepsilon p b^{p-1} s^p \lvert z_0(s)\rvert}{a} \, ds = \frac{\varepsilon pb^{p-1}}{a^{q+1}} \, r^{-q} \int_0^r s^{p} \lvert z_0(s)\rvert \,ds.$$

Finally for term $III$ in \eqref{IIItermlip} we obtain 
$$ III(r) = (ar)^{-q} \int_0^r (bs)^p \, \frac{\lvert z_0(s)\rvert}{\rho(s)\zeta(s)} \big\lvert \zeta(s)-\rho(s)\big\rvert \,ds 
\le \frac{\varepsilon b^p}{a^{q+2}} \, r^{-q} \int_0^r s^{p} \lvert z_0(s)\rvert \, ds.$$
Summing up, we obtain 
$$ \big\lvert \vectyFone(\rho)(r) - \vectyFone(\zeta)(r)\big\rvert \le 
\frac{\varepsilon b^{p-1} }{a^{2q+1}}  \,  \big( q b^q  + p a^q +  b a^{q-1} \big)    r^{-q}  \int_0^r s^{p} \lvert z_0(s)\rvert \, ds.$$
Hence 
$$ \lVert \vectyFone(\rho)-\vectyFone(\zeta)\rVert \le C\varepsilon r^{-q} \int_0^r s^{p+n-1} \lvert \omega_0(s)\rvert \, ds \le 
C\varepsilon \int_0^{\infty} s^{p-q+n-1} \lvert \omega_0(s)\rvert \, ds,$$
as desired.

The computations for the term $\vectyFtwo$ are exactly the same, except that the integrals are on $[r,\infty)$ and the estimate becomes 
$$ \lVert \vectyFtwo(\rho)-\vectyFtwo(\zeta)\rVert \le C \varepsilon r^p \int_r^{\infty} s^{n-1-q} \lvert \omega_0(s)\rvert \,ds \le 
C\varepsilon \int_0^{\infty} s^{p-q+n-1} \lvert \omega_0(s)\rvert \, ds,$$
using the fact that $r^p\le s^p$ for $s\ge r$, since $p\ge 0$. 
This is the same as the estimate for $\vectyFone$. 
\end{proof}

We now conclude local well-posedness of the vector field \eqref{rhovecpowers}, and thus local well-posedness of the $\dot{H}^1$ and $\dot{H}^2$ Euler-Arnold 
equations in the form \eqref{rhoODE} on the space $\spacey$ defined in \eqref{spaceydef}.

\begin{corollary}\label{localexistenceC0}
Consider the vector field $X$ given by \eqref{vectorfield} and the differential equation $\frac{d\rho}{dt} = X(\rho)$ on the space $\spacey$ defined by \eqref{spaceydef}, where $X(\rho) = \rho\vectyF(\rho)$  and $\vectyF$ given by \eqref{vecty}--\eqref{vecfields12}. Then $X$ is locally Lipschitz on $\spacey$, and the solution can be constructed on a time interval $[0,T)$ as long as $\rho(t)$ remains bounded above and below by positive 
numbers. In particular the radial $\dot{H}^1$ and $\dot{H}^2$ equations \eqref{rhoODE} for $\rho$ with $\delta$ given by either \eqref{H1dotinversion} or \eqref{H2dotinversion} are locally well-posed on the space $\spacey$, as long as either 
$$ \int_0^{\infty} \lvert \omega_0(s)\rvert \,ds < \infty \qquad \text{or}\qquad \int_0^{\infty} s^2 \lvert \omega_0(s)\rvert \,ds < \infty,\qquad \text{respectively}. $$
\end{corollary}

\begin{proof}
Proposition \ref{rhopowerslipschitzthm} shows that on any open ball $\Pspace_{a,b}$, the fields $\vectyFone_{p,q+1}$ and $\vectyFtwo_{p-1,q}$ are both bounded Lipschitz vector fields for $p\ge 1$ and $q\ge 0$, as long as 
\begin{equation}\label{momentumL1}
\int_0^{\infty} s^{p+n-q-2} \lvert \omega_0(s)\rvert \, ds < \infty.
\end{equation}
Since our given vector field $\rho \vectyF$ is a product of $\rho$ with a linear combination sum of $\vectyFone_{p,q+1}$ and $\vectyFtwo_{p-1,q}$, it is obviously also bounded and locally Lipschitz under the same condition.

In the $\dot{H}^1$ case we have a single term of the form \eqref{rhovecpowers} with $p=1$ and $q=n-1$, so condition \eqref{momentumL1} is given by 
$$ \int_0^{\infty} \lvert \omega_0(s)\rvert \, ds < \infty.$$
In the $\dot{H}^2$ case we have two terms of this form, with either $p=3$ and $q=n-1$ or $p=1$ and $q=n-3$, and in either case the relevant combination is $p+n-q-2 = 2$, so condition \eqref{momentumL1} is given by 
$$ \int_0^{\infty} s^2 \lvert \omega_0(s)\rvert \, ds < \infty.$$
\end{proof}

The computation above can be used to show not only that $\vectyF$ is locally Lipschitz, but is in fact continuously differentiable of all orders, and likely even real analytic. One proceeds as follows: the derivative of $\vectyFone_{p,q}$ with respect to $\rho$ in a direction $v$ is given by 
\begin{multline*}
(D\vectyFone_{p,q})_{\rho}(v)(r) %
=-q \gamma(r)^{-q-1} \eta(r) \int_0^r \frac{\gamma(s)^p }{\rho(s)}  \, z_0(s)\,ds \\
+ p\gamma(r)^{-q} \int_0^r \frac{\gamma(s)^{p-1} \eta(s) }{\rho(s)}  \, z_0(s)\,ds 
- \gamma(r)^{-q} \int_0^r \frac{\gamma(s)^p v(s)}{\rho(s)^2}  \, z_0(s)\,ds,
\end{multline*}
where $\eta(r) = \int_0^r v(s)\,ds$. This is of the same basic form as $\vectyFone_{p,q}(\rho)$ itself, with the same powers inside and outside the integral, since $\gamma(r)$ and $\eta(r)$ are both bounded above and below by a constant times $r$. Hence the Lipschitz estimate works here as well, and we get that $\vectyFone$ is continuously differentiable, of the same form as we started. We can then iterate this procedure inductively and obtain a $C^{\infty}$ vector field on $\Pspace$, and keeping track of the size of the bounds on the derivatives would also be expected to give real analyticity, as in \cite{washabaugh2016sqg}.

The technique used here is also applicable for the nonhomogeneous $H^1$ and $H^2$ metrics, but the definition of $\spacey$ and the estimates become rather more complicated, 
and would take us too far afield. Already in the simplest case, $H^1$ in dimension $n=1$, we have the Camassa-Holm equation where a similar analysis was performed in \cite{lee2017local}, requiring a more careful definition of the space $\spacey$ to make everything work. The difference is that here $\partial_1\delta(r,s)$ is essentially the same as $\delta(r,s)/r$, while Bessel functions grow at infinity like exponentials so that $\partial_1\delta(r,s)$ looks like $\delta(r,s)$. Since we have a general local existence result in $H^{\infty}$ for the full $H^k$ equations already (see Proposition \ref{prop:wellposedness}), this result is not necessary for the present purpose of studying global existence, but it would be interesting to establish local existence in the best-possible space of $C^1$ radial diffeomorphisms for the nonhomogeneous metrics on $\mathbb{R}^n$.

\bibliography{refs}
\bibliographystyle{abbrv}

\end{document}